\documentclass[final]{siamltex}

\usepackage{latexsym, amssymb, amsmath}
\usepackage{cancel}
\usepackage{graphicx}
\usepackage{mathtools}
\usepackage{float}
\usepackage{chngcntr}
\usepackage{moreverb}
\usepackage{mathdots}
\usepackage{color}
\usepackage{enumitem}

\newtheorem{example}[theorem]{Example}
\let\oldexample\example
\renewcommand{\example}{\oldexample\normalfont}
\newtheorem{remark}[theorem]{Remark}
\let\oldremark\remark
\renewcommand{\remark}{\oldremark\normalfont}
\usepackage{mathdots}

\title{Conditioning  and backward errors of eigenvalues of homogeneous matrix polynomials \\ under M\"{o}bius transformations\thanks{{ This work was partially supported by the Ministerio de Econom\'ia, Industria y Competitividad (MINECO) of Spain through grants MTM2012-32542, MTM2015-65798-P, and MTM2017-90682-REDT. The research of L.M. Anguas is funded by the ``contrato predoctoral'' BES-2013-065688 of MINECO.}}}

\author{Luis Miguel  Anguas\thanks{Departamento de Matem\'aticas, Universidad Carlos III de Madrid,
		Avda.\ Universidad 30, 28911 Legan\'es, Spain ({\tt languas@math.uc3m.es})} \and Maria Isabel Bueno\thanks{Department of Mathematics and College of Creative Studies,
		University of California, Santa Barbara, CA 93106, USA ({\tt mbueno@math.ucsb.edu}){ .}} \and Froil\'an M.  Dopico\thanks{Departamento de Matem\'aticas, Universidad Carlos III de Madrid,
		Avda.\ Universidad 30, 28911 Legan\'es, Spain ({\tt dopico@math.uc3m.es})}}

\begin{document}
\maketitle
\begin{abstract}
	M\"{o}bius transformations have been used in numerical algorithms for computing eigenvalues and invariant subspaces of structured generalized and polynomial eigenvalue problems (PEPs). These transformations convert problems with certain structures arising in applications into problems with other structures and whose eigenvalues and invariant subspaces are easily related to the ones of the original problem. Thus, an algorithm that is efficient and stable for some particular structure can be used for solving efficiently another type of structured problem via an adequate M\"{o}bius transformation. A key question in this context is whether these transformations may change significantly the conditioning of the problem and the backward errors of the computed solutions, since, in that case, their use might lead to unreliable results. We present the first general study on the effect of M\"{o}bius transformations on the eigenvalue condition numbers and backward errors of approximate eigenpairs of PEPs. By using the homogeneous formulation of PEPs, we are able to obtain two clear and simple results. First, we show that, if the matrix inducing the M\"{o}bius transformation is well conditioned, then such transformation approximately preserves the eigenvalue condition numbers and backward errors when they are defined with respect to perturbations of the matrix polynomial which are small relative to the norm of the whole polynomial. However, if 
	the perturbations in each coefficient of the matrix polynomial are small relative to the norm of that coefficient, then the corresponding eigenvalue condition numbers and backward errors are preserved approximately by the M\"{o}bius transformations induced by well-conditioned matrices only if a penalty factor, depending on the norms of those matrix coefficients, is moderate. It is important to note that these simple results are no longer true if a non-homogeneous formulation of the PEP is used.
	
\end{abstract}


\section{Introduction}\label{sect.intro} 
M\"{o}bius transformations are a standard tool in the theory of matrix polynomials and in their applications. The use of M\"{o}bius transformations of matrix polynomials can be traced back to  at least   \cite{MCM1,MCM2}, where they are defined for general rational matrices which are not necessarily polynomials. Since M\"{o}bius transformations change the eigenvalues of a matrix polynomial in a simple way and preserve most of the properties of the polynomial \cite{Mobius}, they have often been used  to transform a matrix polynomial with infinite eigenvalues into another polynomial with only finite eigenvalues and for which a certain problem can be solved more easily. Recent examples of this theoretical use can be found, for instance, in \cite{dtdopvd2015,TisZab}.

A fundamental property of some M\"{o}bius transformations, called Cayley transformations, is to convert matrix polynomials with certain structures arising in control applications into matrix polynomials with other structures that also arise in applications. This allows to translate many properties from one structured class of matrix polynomials into another. The origins of these results on structured problems are found in classical group theory, where Cayley transformations are used, for instance, to transform Hamiltonian into symplectic matrices and vice versa \cite{weyl}. Such results were extended to Hamiltonian and symplectic matrix pencils, i.e., matrix polynomials of degree one, in \cite{Mehr91,cayley} (with the goal of relating discrete and continuous control problems) and generalized to several classes of structured matrix polynomials of degree larger than one in \cite{good}. A thorough treatment of  the properties of M\"{o}bius transformations of matrix polynomials is presented in a unified way in \cite{Mobius}.

The Cayley transformations mentioned in the previous paragraph are not just of theoretical interest, since they, and some variants, have been used explicitly in a number of important numerical algorithms for eigenvalue problems. Some examples are: \cite[Algorithm 4.1]{benneretal}, where they are used for computing the eigenvalues of a sympletic pencil by transforming such pencil into a Hamiltonian pencil, and then using a structured eigenvalue algorithm for Hamiltonian pencils; \cite[Sec. 3.2]{MehrPol}, where they are used to transform a matrix pencil into another one so that the eigenvalues in the left-half plane of the original pencil are moved into the unit disk, an operation that is a preprocessing before applying an inverse-free disk function method for computing certain stable/un-stable deflating subspaces of the matrix pencil; and \cite[Sec. 6]{MehrXu}, where they are used for transforming palindromic/anti-palindromic pencils into even/odd pencils with the goal of deflating the $\pm 1$ eigenvalues  of the palindromic/anti-palindromic pencils via algorithms for deflating the infinite eigenvalues of the even/odd pencils. Other examples can be found in the literature, although, sometimes, the use of the Cayley transformations is not mentioned explicitly. For instance, the algorithm in \cite{byers} for computing the structured staircase form of skew-symmetric/symmetric pencils can be used via a Cayley transformation and its inverse for computing a structured staircase form of palindromic pencils, although this is not mentioned in \cite{byers}.

The numerical use of M\"{o}bius transformations in structured algorithms for pencils discussed in the previous paragraph can be extended to matrix polynomials of degree larger than one. Assume that a structured matrix polynomial $P$ is given and we want to solve the corresponding polynomial eigenvalue problem (PEP). Then, the standard procedure is to consider  one of the (strong) linearizations $L$ of $P$ of the same structure  available in the literature (see, for instance, \cite{buenoetal,dtdomack,good}), assuming it exists. Assume also that a backward stable structured algorithm is available for a certain type of structured pencils and that $L$ can be transformed into a pencil with such structure through a M\"obius transformation, $M_A$. By \cite[Corollary 8.6]{Mobius}, $M_A(L)$ is a (strong) linearization of $M_A(P)$. However, even if the structured algorithm guarantees that  the PEP associated with $M_A(P)$ is solved in a backward stable way \cite{dopevd}, it is not guaranteed that it solves the PEP associated with $P$ in a backward stable way as well. Thus, a direct way of checking if this is the case is to analyze how the M\"obius transformation affects the backward errors of the computed eigenpairs of the polynomial $P$.

As illustrated in the previous paragraph, when the numerical solution of a problem is obtained by transforming the problem into another one, a fundamental question is whether or not such transformation deteriorates the conditioning of the problem and/or the backward errors of the approximate solutions, because a significant deterioration of such quantities may lead to unreliable solutions. We have not found in the literature any  analysis of this kind concerning the use of M\"{o}bius transformations for solving PEPs, apart from a few vague comments in some papers. The results in this paper are a first step in this direction. More specifically, we present the first general study on the effect of M\"{o}bius transformations on the eigenvalue condition numbers and backward errors of approximate eigenpairs of PEPs. We are aware that this analysis does not cover all the numerical applications of M\"{o}bius transformations that can be found in the literature, since, for instance, the effect on the conditioning of the deflating subspaces of pencils is not covered in our study. For brevity, this and other related problems will be considered in future works.

In this paper, the PEP is formulated in homogeneous form \cite{Ded-first,DedTis2003,Stewart} and the corresponding homogeneous eigenvalue condition numbers \cite{DedTis2003,comparison} and backward errors \cite{HigLiTis} are used. This homogeneous formulation has clear mathematical advantages  over the standard non-homogeneous one \cite{DedTis2003,comparison} and has been used recently in the analysis of algorithms for solving PEPs via linearizations \cite{hamarling, HigMacTis}. In addition, when a PEP is solved by applying the QZ algorithm to a linearization of the corresponding matrix polynomial, the computed eigenvalues are, in fact, the homogeneous eigenvalues. Note that the non-homogeneous eigenvalues are obtained from the homogeneous ones  by the division of its two components,  and this operation is performed only after the algorithm QZ has converged. The analysis of the effect of M\"{o}bius transformations on the eigenvalue condition numbers of non-homogeneous matrix polynomials is postponed to a future paper for brevity, but also because it is cumbersome since it requires to distinguish several cases. Such complications are related to the fact that, for any M\"{o}bius transformation, it is possible to find matrix polynomials for which the modulus of some of its non-homogenous eigenvalues changes wildly under the transformation. This fact has led to the popular belief that any M\"{o}bius transformation affects dramatically the conditioning of certain critical eigenvalues, something that is not true in the homogeneous formulation. 

By using the homogeneous formulation of PEPs, we are able to obtain, among many others, two clear and simple results that are highlighted in the next lines. First, we show in Theorems \ref{main-homo1} and \ref{eta-bounds} that, if the matrix inducing the M\"{o}bius transformation is well conditioned, then, for any matrix polynomial and simple eigenvalue, such transformation approximately preserves the eigenvalue condition numbers and backward errors when, in the definition of these magnitudes,  small perturbations of the matrix polynomial relative to the norm of the whole polynomial are considered. However, if we consider condition numbers and backward errors for which 
the perturbations in each coefficient of the matrix polynomial are small relative to the norm of that coefficient, then these magnitudes are approximately preserved  by the M\"{o}bius transformations induced by well-conditioned matrices only if a penalty factor, depending on the norms of the coefficients of the polynomial, is moderate. This is proven in Theorems \ref{main-homo} and \ref{eta-bounds}.

The paper is organized as follows. In Section \ref{sec-notation}, we introduce some notation and basic definitions about matrix polynomials. Section \ref{Sec-Mobius} contains results about M\"{o}bius transformations of homogeneous matrix polynomials. Most of such results are well-known, but the ones in Subsection \ref{subsec.mobius} are new, as far as we know. In Section \ref{cond-back-sect}, we recall the definitions and expressions of eigenvalue condition numbers and backward errors of PEPs. Sections \ref{SecHom} and \ref{sec.backwerrors}
include the most important results proven in this paper about the effect of M\"{o}bius transformations on eigenvalue condition numbers and backward errors of approximate eigenpairs of PEPs. Numerical experiments that illustrate the theoretical results in previous sections are described in Section \ref{sec:num}. Finally, Section \ref{sec.conclusions} discusses the conclusions and some lines of future research.


\section{Notation and basic definitions} \label{sec-notation}
To begin with, let us introduce some general notation that will be used throughout this paper. Given positive integers $a$ and $b$, we define
$$a:b := \left \{ \begin{array}{ll} a, a+1, \ldots, b, & \textrm{if $a\leq b$,} \\ \emptyset, & \textrm{if $a> b$.}\end{array} \right.$$
For any real number $\alpha$, $\lfloor \alpha \rfloor$ denotes the largest integer smaller than or equal to $\alpha$. The field of complex numbers is denoted by $\mathbb{C}$.
For any complex vector $x=[x_1, \ldots, x_n]^T \in \mathbb{C}^n$, $\|x\|_p$ denotes its $p$-norm, i.e., $\|x\|_p := (\sum_{i=1}^n |x_i|^p)^{1/p}$, for $1 \leq p < \infty$. We also denote $\|x\|_{\infty}:=\max_{i=1:n}\{ |x_i|\}$. For any complex matrix $A\in\mathbb{C}^{m\times n}$, $\|A\|_2$ denotes its spectral or 2-norm, that is, its largest singular value;  
$\|A\|_\infty$ denotes its $\infty$-norm, that is, the maximum row sum of the moduli of its entries;  and $\|A\|_1$ denotes its 1-norm, that is, the maximum column sum of the moduli of its entries. Additionally,  $\|A\|_M:=\max\{|A_{ij}|, i=1:m, j=1:n\}$ denotes the max norm of $A$.

Let us present now some basic concepts that will be used frequently in this paper.

A matrix polynomial $P(\alpha,\beta)$ is said to be a \emph{homogeneous matrix polynomial of degree $k$}  if it is of the form
\begin{equation}\label{homo-pol}
P(\alpha, \beta)=\sum_{i=0}^k \alpha^i \beta^{k-i} B_i,\quad B_i\in\mathbb{C}^{m\times n}{ ,}
\end{equation}
where all the matrix coefficients $B_i$ but one are allowed to be zero. If all matrix coefficients $B_i$ are zero, then we say that $P$ has degree $-\infty$ or it is undefined.  If $m=n$ and the determinant of $P(\alpha,\beta)$ is not identically equal to zero, $P$ is said to be \emph{regular}. Otherwise, it is said to be \emph{singular}. 

Given a regular homogeneous matrix polynomial $P(\alpha,\beta)$, the \emph{(homogeneous) polynomial eigenvalue problem} (PEP) associated to $P(\alpha, \beta)$ consists of finding pairs of scalars $(\alpha_0,\beta_0)\neq (0,0)$ and nonzero vectors $x,y\in\mathbb{C}^{n}$ such that
\begin{equation}
\label{homo-PEP}
y^*P(\alpha_0,\beta_0)=0 \quad \text{and}\quad P(\alpha_0,\beta_0)x=0.
\end{equation}
{The pair} $(\alpha_0, \beta_0)$ is called an \emph{eigenvalue} of $P(\alpha, \beta)$, the vectors $x$ and $y$ are called, respectively, a \emph{right and a left eigenvector} of $P(\alpha, \beta)$ associated with $(\alpha_0, \beta_0)$, and the pairs $(x, (\alpha_0, \beta_0))$ and $(y^*, (\alpha_0, \beta_0))$ are called, respectively, a right and a left eigenpair of $P(\alpha, \beta)$. 

Note that, {for any complex number $a\neq 0$}, the pair  $ (a \alpha_0, a\beta_0)$ is an eigenvalue of $P(\alpha, \beta)$ if and only if $(\alpha_0, \beta_0)$ is an eigenvalue of $P(\alpha, \beta)$. Thus, an eigenvalue of a  matrix polynomial $P(\alpha, \beta)$ can be seen as a line in $\mathbb{C}^2$ passing through the origin whose points are solutions to the equation $\textrm{det}(P(\alpha, \beta))=0$. Throughout the paper, we denote eigenvalues, i.e., lines, as pairs $(\alpha_0, \beta_0)$ and a specific (nonzero) representative of this eigenvalue, i.e., a specific (nonzero) point on the line $(\alpha_0, \beta_0)$ in $\mathbb{C}^2$, by $[\alpha_0, \beta_0]^T$. We will also use the notation  $\langle x \rangle$, where $x\in \mathbb{C}^2$,  to denote the line generated by the vector $x$ in $\mathbb{C}^2$ through scalar multiplication. In particular, $\langle [\alpha_0, \beta_0]^T \rangle = (\alpha_0, \beta_0).$
Notice that all representatives of an eigenvalue of $P(\alpha, \beta)$ are nonzero scalar multiples of each other. 

In future sections, we will need to calculate the norm of  a homogeneous matrix polynomial $P(\alpha, \beta)$ as in \eqref{homo-pol}.  In this paper we will use the norm $\|P\|_\infty:=\max_{i=0:k}\{\|B_i\|_2\}$.

Some of the results for homogeneous matrix polynomials that will be introduced in Section \ref{Sec-Mobius} were proven in the literature for \emph{non-homogeneous matrix polynomials}, that is, matrix polynomials {written in} the form
\begin{equation}
\label{nonhom-pol}
P(\lambda)=\sum_{i=0}^k\lambda^i B_i, \quad B_i\in\mathbb{C}^{m\times n}.
\end{equation}
{To extend} those results to homogeneous matrix polynomials it will be enough to notice that a homogeneous matrix polynomial $P(\alpha,\beta)$ can be rewritten in { non-homogeneous} form as follows
\begin{equation}
\label{rel-hom-non}
P(\alpha, \beta)= \left \{ \begin{array}{ll} \beta^k P(\alpha/\beta), & \textrm{if $\beta\neq 0$,}\\
\alpha^k B_k, & \textrm{if $\beta=0$.} \end{array} \right.
\end{equation}

When $n=m$, we say that a non-homogeneous matrix polynomial $P(\lambda)$ is regular if its determinant is not identically zero. In this case, we can consider the non-homogeneous PEP associated to  $P(\lambda)$. As in the homogeneous case, it consists of  finding scalars $\lambda_0$ and nonzero vectors $x$ and $y$ { such that $y^*P(\lambda_0)=0$ and $P(\lambda_0)x=0.$
	The vectors $x$ and $y$ are said to be right} and left eigenvectors of $P(\lambda)$ associated with the eigenvalue $\lambda_0$, and the pairs $(x, \lambda_0)$ and $(y^*, \lambda_0)$ are called, respectively, a right and a left eigenpair of $P(\lambda)$. 

The next lemma provides a relationship between the eigenvalues and eigenvectors of a matrix polynomial when expressed in homogeneous and non-homogeneous { forms}.  We omit its proof because it is straightforward.
\begin{lemma}
	\label{rel-hom-nonhom-PEP}
	A pair $(x, (\alpha_0, \beta_0))$ (resp. $(y^*, (\alpha_0, \beta_0) )$) is a right (resp. left) eigenpair for a { regular} homogeneous matrix polynomial $P(\alpha, \beta)$ if and only if $(x, \lambda_0)$ (resp. $(y^*, \lambda_0))$ is a right (resp. left) eigenpair for the same polynomial when expressed in non-homogeneous form,  where $\lambda_0=\alpha_0/\beta_0$ if $\beta_0 \neq 0$ and  $\lambda_0=\infty$ if $\beta_0=0$.
\end{lemma}

\section{M\"{o}bius transformations of homogeneous matrix polynomials}
\label{Sec-Mobius}  
Before introducing the definition of M\"obius transformation of matrix polynomials, we present some notation that will be used in this section. We denote by $GL(2, \mathbb{C})$ the set of nonsingular $2\times 2$ matrices with complex entries and by $\mathbb{C}[\alpha, \beta]_k^{m\times n}$ the vector space of $m\times n$ homogeneous matrix polynomials of degree $k$ whose matrix coefficients have complex entries together with the zero polynomial, that is, polynomials of the form (\ref{homo-pol}) whose matrix coefficients are allowed to be all zero.

Next we introduce the concept of M\"{o}bius transformation.
\begin{definition}\cite{Mobius}
	Let $A=\left[ \begin{array}{cc} a & b \\ c & d \end{array} \right] \in GL(2,\mathbb{C})$. Then the M\"{o}bius transformation on $\mathbb{C}[\alpha, \beta]_k^{m\times n}$ induced by $A$ is the map $M_A: \mathbb{C}[\alpha, \beta]^{m\times n}_{k} \rightarrow \mathbb{C}[\alpha, \beta]^{m\times n}_{k}$ given   by
	\begin{equation}\label{mobiusP}
	M_A\left(\sum_{i=0}^k \alpha^i \beta^{k-i} B_i\right)(\gamma, \delta)=\sum_{i=0}^k(a\gamma+b\delta)^i(c\gamma+d\delta)^{k-i}B_i.
	\end{equation}
	We call M\"{o}bius transform of $P(\alpha,\beta)$ under $M_A$ to the matrix polynomial $M_A(P)(\gamma,\delta)$, that is, the image of $P(\alpha, \beta)$  under $M_A$.
\end{definition}

It is important to highlight that the M\"{o}bius {transform} of a homogeneous matrix polynomial $P$ of degree $k$ is another homogeneous matrix polynomial of the same degree.  

The next example shows that the well-known reversal of a {matrix polynomial $P(\alpha, \beta)$ \cite{Mobius} can be seen as a M\"obius transform} of $P$. 

\begin{example}
	\label{ex-reversal}
	Let us consider the M\"obius transformation induced by the matrix $R=\left[ \begin{array}{cc} 0 & 1 \\ 1 & 0 \end{array} \right].$ Given $P(\alpha, \beta)=\sum_{i=0}^k \alpha^i \beta^{k-i}B_i$, we have
	$$M_R(P)(\gamma, \delta)= \sum_{i=0}^k  \gamma^{k-i} \delta^{i} B_i = \textrm{rev}\; P (\gamma, \delta).$$ 
\end{example}


In the next definition, we introduce some M\"obius transformations that  are useful in converting some types of structured matrix polynomials into others { \cite{lan-rod, good, Mobius, cayley}.}

\begin{definition}\label{cayley}
	The M\"{o}bius transformations induced by the matrices
	\begin{equation}\label{cayley-trans}
	A_{+1}=\left[ \begin{array}{rr} 1 & 1 \\ -1 & 1 \end{array} \right ], \quad A_{-1}=\left[ \begin{array}{rr} 1 & -1 \\ 1 & 1 \end{array} \right]
	\end{equation}
	are called Cayley transformations. 
\end{definition}


\subsection{Properties of M\"obius transformations}

In this section, we present some properties of the M\"obius transformations that were proven in \cite{Mobius} for  non-homogeneous matrix polynomials, that is, matrix polynomials of the form \eqref{nonhom-pol}.  The proof of the equivalent statement of those properties for homogeneous polynomials follows immediately from the results in \cite{Mobius} and the relationship \eqref{rel-hom-non} between the homogeneous and non-homogeneous expressions of the same matrix polynomial.

\begin{proposition}{\rm\cite[Proposition 3.5]{Mobius}}
	For any $A\in GL(2,\mathbb{C})$, $M_A$ is a $\mathbb{C}$-linear operator on the vector space $\mathbb{C}[\alpha, \beta]_k^{m\times n}$, that is, for any $\mu\in\mathbb{C}$ and $P,Q\in \mathbb{C}[\alpha, \beta]_k^{m\times n}$,
	\begin{center}
		$M_A(P+Q)=M_A(P)+M_A(Q)$ \ \text{and} \ $M_A(\mu P)=\mu M_A(P)$.
	\end{center}
	
\end{proposition}

\begin{proposition}\label{prop-mobius}{\rm\cite[Theorem 3.18 and Proposition 3.27]{Mobius}}
	Let  $A, B \in GL(2, \mathbb{C})$ and let $I_2$ denote the $2 \times 2$ identity matrix. Then, when  the M\"obius transformations are seen as operators on $\mathbb{C}[\alpha, \beta]^{m\times n}_k$,  the following properties hold:
	\begin{enumerate}
		\item $M_{I_2}$ is the identity operator; 
		\item $M_A \circ M_B = M_{BA}$;
		\item $(M_A)^{-1}=M_{A^{-1}}$;
		\item $M_{\mu A} = \mu^k M_A$, for any nonzero $\mu \in \mathbb{C}$; \item {If $m=n$, then  $\mathrm{det}(M_A(P)) = M_A(\mathrm{det}(P))$, where the M\"obius transformation on the right-hand side operates on $\mathbb{C}[\alpha, \beta]^{1\times 1}_{nk}$.}
	\end{enumerate}
\end{proposition}

\begin{remark}\label{regular}
	An immediate consequence of Proposition \ref{prop-mobius}(5.) is that $P$ is a regular matrix polynomial if and only if $M_A(P)$ is.
\end{remark}

{ The following result provides a connection between the eigenpairs of a regular homogeneous matrix polynomial $P(\alpha, \beta)$ and the eigenpairs of a M\"obius transform $M_A(P)(\gamma, \delta)$ of $P(\alpha, \beta)$.
	As the previous properties, this result was proven for non-homogeneous matrix polynomials in \cite{Mobius}. It is easy to see that an analogous result follows when $P$ is expressed in homogeneous form using (\ref{rel-hom-non}) and Lemma \ref{rel-hom-nonhom-PEP}.}

\begin{lemma}\label{eig-hom}{\rm\cite[Remark 6.12 and Theorem 5.3]{Mobius}}
	Let $P(\alpha, \beta)$ be a regular homogeneous matrix polynomial and let $A=\left[\begin{array}{cc} a & b \\ c & d\end{array} \right] \in GL(2, \mathbb{C})$.  If 
	$(x, (\alpha_0,\beta_0))$ (resp. $ (y^*, (\alpha_0, \beta_0 )$) is a right (resp. left) eigenpair of $P(\alpha,\beta)$, then  $(x, \langle A^{-1}[\alpha_0, \beta_0]^T\rangle)$ (resp. $(y^*, \langle A^{-1}[\alpha_0, \beta_0]^T \rangle)$)  is a right (resp. left) eigenpair of $M_A(P)(\gamma,\delta)$.
	Moreover, $(\alpha_0, \beta_0)$, as an eigenvalue of $P(\alpha, \beta)$, has the same algebraic multiplicity as $\langle A^{-1}[\alpha_0, \beta_0]^T\rangle$, when considered an eigenvalue of $M_A(P)(\gamma, \delta)$. {In particular, $(\alpha_0, \beta_0)$ is a simple eigenvalue of $P(\alpha, \beta)$ if and only if $\langle A^{-1}[\alpha_0, \beta_0]^T\rangle$  is a simple eigenvalue of $M_A(P)(\gamma, \delta)$.}
\end{lemma}

Motivated by the previous result, we introduce the following definition.

\begin{definition} \label{def.associated}
	Let $P(\alpha, \beta)$ be a regular homogeneous matrix polynomial and let $A=\left[\begin{array}{cc} a & b \\ c & d\end{array} \right] \in GL(2, \mathbb{C})$. Let $(\alpha_0, \beta_0)$ be an eigenvalue of $P(\alpha, \beta)$ and let $[\alpha_0, \beta_0]^T$ be a representative of $(\alpha_0, \beta_0)$. Then, we 
	call $\langle A^{-1}[\alpha_0, \beta_0]^T\rangle$  the \emph{eigenvalue of $M_A(P)$ associated with the eigenvalue $(\alpha_0, \beta_0)$ of $P(\alpha, \beta)$}
	and we call  $ A^{-1}[\alpha_0, \beta_0]^T $ the \emph{representative of the eigenvalue of $M_A(P)$  associated with $[\alpha_0, \beta_0]^T$}.
	
\end{definition}

In the following remark we explain how to compute an explicit expression for the components of the vector $A^{-1}[\alpha_0, \beta_0]^T$.

\begin{remark}
	We recall that, for $A =\left[\begin{array}{cc} a & b \\ c & d\end{array} \right]\in GL(2, \mathbb{C})$,
	\begin{equation}\label{agblfor}A^{-1}= \frac{\mathrm{adj}(A)}{\mathrm{det}(A)},\end{equation}
	where $\mathrm{adj}(A)$ denotes the  adjugate  of the matrix $A$, given by
	$$\mathrm{adj}(A):=  \left[\begin{array}{rr} d & -b \\ -c & a \end{array}\right].$$
	Thus, given a simple eigenvalue $(\alpha_0, \beta_0)$ of a homogeneous matrix polynomial $P$ and a representative $[\alpha_0, \beta_0]^T$ of $(\alpha_0, \beta_0)$,  the components of the representative $[\gamma_0, \delta_0]^T:=A^{-1}[\alpha_0, \beta_0]^T $ of the eigenvalue of $M_A(P)$ associated with $[\alpha_0, \beta_0]^T$ are given by
	\begin{equation}\label{rel-eig}
	\gamma_0:=\frac{d\alpha_0-b \beta_0}{\mathrm{det}(A)}, \quad \delta_0:=\frac{a\beta_0-c\alpha_0}{\mathrm{det}(A)}.
	\end{equation}
	
\end{remark}




The following fact, which follows taking into account that $\|\mathrm{adj}(A)\|_{\infty} = \|A\|_1$ and $\|\mathrm{adj}(A)\|_{1} = \|A\|_{\infty}$,  will be used to simplify the bounds on the quotients of condition numbers presented  in Section \ref{SecHom}: 
\begin{equation}\label{quotient-norm-eig} \frac{1}{|\mathrm{det}(A)|} = \frac{\|A^{-1}\|_{\infty}}{\|A\|_{1}}= \frac{\|A^{-1}\|_{1}}{\|A\|_{\infty}}.
\end{equation}

\subsection{The matrix coefficients of the M\"obius transform of a matrix polynomial}
\label{subsec.mobius}
When comparing the condition number of an eigenvalue  of a regular homogeneous matrix polynomial $P(\alpha, \beta)$ with the condition number of the associated eigenvalue  of  the  M\"obius transform $M_A(P)$ of $P$, it will be useful to have an explicit expression for the coefficients of the matrix polynomial $M_A(P)$ in terms of the matrix coefficients of $P$, as well as an upper bound on the 2-norm of each coefficient of $M_A(P)$ in terms of the norms of the coefficients of $P$. We provide such expression and upper bound in the following proposition. 

\begin{proposition}\label{cotaMobius}
	Let $P(\alpha, \beta)=\sum_{i=0}^k \alpha^i \beta^{k-i} B_i \, {\in \mathbb{C}[\alpha , \beta]^{m\times n}_k}$, $A=\left[ \begin{array}{cc} a & b \\ c & d \end{array} \right] \in GL(2,\mathbb{C})$, and $M_A$ be the M\"{o}bius transformation induced by $A$ on $\mathbb{C}[\alpha , \beta]^{m\times n}_k$. Then, $M_A(P)(\gamma, \delta)=\sum_{\ell=0}^k \gamma^{\ell}  \delta^{k-\ell} \widetilde{B}_{\ell},$ where 
	\begin{equation}\label{Bl}
	\widetilde{B}_{\ell}=\sum_{i=0}^k\sum_{j=0}^{k-\ell} {i \choose j} {k-i \choose k-j-\ell}a^{i-j}b^jc^{j+\ell-i}d^{k-j-\ell}B_i, \quad \ell=0: k,
	\end{equation}
	and $\displaystyle {s \choose t} :=0$ for $s < t$. Moreover,
	\begin{align}
	\|\widetilde{B}_{\ell}\|_2 &  \leq \|A\|^k_{\infty} {k \choose \lfloor k/2 \rfloor} \sum_{i=0}^k\|B_i\|_2, \quad 
	\ell=0: k.\label{normBi}
	\end{align}
	
\end{proposition}

\begin{proof}
	By the Binomial Theorem,  
	$$(a\gamma+b\delta)^i=\sum_{j=0}^i \left(\begin{array}{c}i\\ j\end{array}\right)(a\gamma)^{i-j}(b\delta)^j, \quad 
	(c\gamma +d\delta)^{k-i}=\sum_{r=0}^{k-i}\left(\begin{array}{c}k-i\\ r\end{array}\right)(c\gamma)^{k-i-r}(d\delta)^r.$$
	Thus, from (\ref{mobiusP}) we get
	\begin{align*}
	M_A(P)(\gamma, \delta)&= \sum_{i=0}^k \sum_{j=0}^i\sum_{r=0}^{k-i}{i \choose j}{k-i \choose r}\gamma^{k-j-r}\delta^{j+r}a^{i-j}b^jc^{k-i-r}d^rB_i\\ 
	&=
	\sum_{i=0}^k\sum_{j=0}^i\sum_{\ell=i-j}^{k-j}{i\choose j}{k-i\choose k-j-\ell}\gamma^{\ell}\delta^{k-\ell} a^{i-j}b^jc^{j+\ell-i}d^{k-j-\ell}B_i \\
	& = \sum_{\ell=0}^k \sum_{i=0}^k \sum_{j=\max\{0, i-\ell\}}^{\min\{i, k-\ell\}}{i \choose j} {k-i \choose k-j-\ell}\gamma^{\ell}\delta^{k-\ell} a^{i-j}b^jc^{j+\ell-i}d^{k-j-\ell}B_i\\
	& = \sum_{\ell=0}^k \sum_{i=0}^k \sum_{j=0}^{k-\ell} {i \choose j} {k-i \choose k-j-\ell}\gamma^{\ell}\delta^{k-\ell} a^{i-j}b^jc^{j+\ell-i}d^{k-j-\ell}B_i,
	\end{align*}
	where the second equality follows by applying the change of variable $\ell=k-j-r,$ and the fourth equality follows because if  $i < k - \ell$ and $j > i$, then ${i \choose j} =0$, and if $i-\ell >0$ and $j < i- \ell$, then ${k-i \choose k- \ell - j} =0.$ Hence, (\ref{Bl}) follows.

	From (\ref{Bl}),  we have { 
		\begin{align*}
		\|\widetilde{B}_{\ell}\|_2
		&\leq
		\sum_{i=0}^k \sum_{j=0}^{k-\ell} {i \choose j} {k-i \choose k-j-\ell}|a|^{i-j}|b|^j|c|^{j+\ell-i}|d|^{k-j-{\ell}}\|B_i\|_2\\
		&\leq \|A\|_M^k \sum_{i=0}^k {k \choose k-l} \|B_i\|_2\leq  \|A\|_{\infty}^k \sum_{i=0}^k {k \choose k-l} \|B_i\|_2, 
		\end{align*} }
	where the second inequality follows from  the Chu-Vandermonde identity \cite{vandermonde}
	\begin{equation}\label{chu}
	\sum_{j=0}^{k-\ell} {i \choose j} {k-i \choose k-\ell-j} ={k \choose k-\ell}.
	\end{equation} 
	The inequality in (\ref{normBi}) follows taking into account that \cite{brualdi}
	$${k \choose k - \ell} \leq {k \choose \lfloor k/2 \rfloor}, \quad 0 \leq \ell \leq k.$$
\end{proof}

\section{Eigenvalue condition numbers and backward errors of matrix polynomials}\label{cond-back-sect}

To measure the change of the condition number of a simple eigenvalue $(\alpha_0,\beta_0)$ of a regular homogeneous matrix polynomial $P(\alpha, \beta)$
when a M\"{o}bius transformation is applied to $P(\alpha, \beta)$, we consider  two eigenvalue condition numbers that have been introduced in the literature. One of them was presented by Dedieu and Tisseur in \cite{DedTis2003} as an alternative to the Wilkinson-like condition number defined for non-homogeneous matrix polynomials in \cite{Tis2000}. We call it the \textit{Dedieu-Tisseur condition number} and denote it by $\kappa_{h}((\alpha_0,\beta_0),P)$.  The other eigenvalue condition number is a generalization to matrix polynomials of the condition number introduced by Stewart and Sun for pencils \cite{Stewart} and we refer to it as the \textit{Stewart-Sun condition number}. We  denote it by $\kappa_{\theta}((\alpha_0,\beta_0),P)$. An advantage of these two eigenvalue condition numbers for homogeneous matrix polynomials over the relative eigenvalue condition number for non-homogeneous matrix polynomials often used in the literature \cite{Tis2000} is that they are well-defined for all simple eigenvalues, including zero and infinity.

We start by introducing the definition of the Stewart-Sun eigenvalue condition number, which is expressed in terms of the chordal distance whose definition we recall next.

\begin{definition}{\rm\cite[Chapter VI, Definition 1.20]{Stewart}}
	\label{def-chord}
	Let $x$ and $y$ be two nonzero vectors in $\mathbb{C}^2$ and let $\langle x \rangle$ and $\langle y \rangle$ denote the lines passing through zero in the direction of $x$ and $y$, respectively.
	The chordal distance between $\langle x \rangle$ and $\langle y \rangle$ is given by $$\chi(\langle x\rangle, \langle y \rangle):= \mathrm{sin}(\theta(\langle x \rangle, \langle y \rangle)),$$
	where $$\theta(\langle x \rangle, \langle y \rangle) := \mathrm{arccos} \frac{|\langle x, y \rangle |}{\|x\|_2\|y\|_2}, \quad 0\leq\theta(\langle x \rangle, \langle y \rangle)\leq\pi/2,$$
	and $\langle x, y \rangle$ denotes the standard Hermitian inner product, i.e., $\langle x,y\rangle=y^*x$.
\end{definition}

\begin{definition}
	\label{cord-cond}
	Let $(\alpha_0, \beta_0)$ be a simple eigenvalue of a regular matrix polynomial
	$P(\alpha, \beta)= \sum_{i=0}^k \alpha^i \beta^{k-i} B_i$ of degree $k$ and let $x$ be a right eigenvector of $P(\alpha, \beta)$ associated with $(\alpha_0, \beta_0)$. We define
	\begin{align*}
	& \kappa_{\theta}((\alpha_0,\beta_0),P):=\lim_{\epsilon\to 0}\sup \bigg\{\frac{\chi((\alpha_0,\beta_0),(\alpha_0+\Delta\alpha_0,\beta_0+\Delta\beta_0))}{\epsilon}:\\
	& [P(\alpha_0+\Delta\alpha_0,\beta_0+\Delta\beta_0)+\Delta P(\alpha_0+\Delta\alpha_0,\beta_0+\Delta\beta_0)](x+\Delta x)=0,\\
	& \|\Delta B_i\|_2\leq\epsilon\;\omega_i, i=0:k\bigg\},
	\end{align*}
	where  $\Delta P(\alpha, \beta)=\sum_{i=0}^k \alpha^i \beta^{k-i}\Delta B_i$ and $\omega_i$, $i=0:k$, are nonnegative weights that allow flexibility in how the perturbations of $P(\alpha, \beta)$ are measured.
\end{definition}

The next theorem  presents an explicit formula for this condition number.

\begin{theorem}\label{teorhomcondnumb}
	{\rm\cite[Theorem 2.13]{comparison}}
	Let $(\alpha_0,\beta_0)$ be a simple eigenvalue of a regular matrix polynomial $P(\alpha,\beta)=\sum_{i=0}^k\alpha^i\beta^{k-i}B_i$, and let  $y$ and $x$ be, respectively, a left and a right eigenvector of $P(\alpha, \beta)$  associated with $(\alpha_0, \beta_0)$. Then, the Stewart-Sun eigenvalue condition number of $(\alpha_0,\beta_0)$ is given by
	\begin{equation}\label{hom-form}
	\kappa_{\theta}((\alpha_0,\beta_0),P)=\left (\sum_{i=0}^k|\alpha_0|^{i}|\beta_0|^{k-i} \omega_i \right) \frac{\|y\|_2\|x\|_2}{|y^*(\overline{\beta_0}D_{\alpha} P(\alpha_0, \beta_0)-\overline{\alpha_0}D_{\beta} P(\alpha_0, \beta_0))x|},
	\end{equation}
	where $D_{z} \equiv \frac{\partial}{\partial z}$ denotes the partial derivative with respect to $z \in \{\alpha , \beta \}$.
\end{theorem}

It is important to note that the explicit expression for  $\kappa_{\theta}((\alpha_0,\beta_0),P)$ does not depend on the choice of representative of the eigenvalue  $(\alpha_0, \beta_0)$.

The definition of the Dedieu-Tisseur condition number is quite involved. For that reason, it is not included in this paper. For the interested reader, that definition can be found in \cite{DedTis2003}. An explicit formula for this eigenvalue condition number is available in \cite[Theorem 4.2]{DedTis2003}.


In \cite[Corollary 3.3]{comparison}, it was proven that the Stewart-Sun and the Dedieu-Tisseur eigenvalue condition numbers differ at most by a factor $\sqrt{k+1}$ and, so, that they are equivalent. Therefore, it is enough to focus on studying the influence of M\"{o}bius transformations on just one of these two  condition numbers. The corresponding results for the other one can be easily obtained from Corollary 3.3 in \cite{comparison}. We focus on the Stewart-Sun condition number in this paper for two reasons: 1) the Stewart-Sun condition number is  easier to define than  the Dedieu-Tisseur condition number and its definition provides a  geometric intuition of the change in the eigenvalue that it measures; 2) in a subsequent paper we will study the effect of M\"obius transformations on the Wilkinson-like condition  number of the simple eigenvalues of a non-homogeneous matrix polynomial \cite{Tis2000}. Theorem 3.5  in  \cite{comparison} provides a simple connection between this non-homogeneous  condition number and the Stewart-Sun condition number that can be used to obtain the results for the non-homogeneous case from the results obtained in this paper for the Stewart-Sun condition number in a more straightforward way.

In the explicit expression for $\kappa_{\theta}((\alpha_0,\beta_0),P)$, the weights $\omega_i$ can be chosen in different ways leading to different variants of this condition number. In the following definition, we introduce the three types of weights (and the corresponding condition numbers) considered in this paper.
\begin{definition}\label{weights} With the same notation and assumptions as in Theorem \ref{teorhomcondnumb}:
	\begin{enumerate}
		\item The \emph{absolute eigenvalue condition number} of $(\alpha_0,\beta_0)$ is defined by taking $\omega_i=1$ for $i=0:k$ in $\kappa_{\theta}((\alpha_0,\beta_0),P)$ and is denoted by $\kappa_{\theta}^{a}((\alpha_0,\beta_0),P)$. 
		\item The \emph{relative with respect to the norm of $P$ eigenvalue condition number} of $(\alpha_0,\beta_0)$ is defined  by taking $\omega_i=\|P\|_\infty = \max \limits_{j=0:k} \{\|B_j\|_2\}$ for $i=0:k$ in $\kappa_{\theta} ((\alpha_0,\beta_0),P)$ and is denoted by $\kappa_{\theta}^{p}((\alpha_0,\beta_0),P)$.   
		\item The \emph{relative eigenvalue condition number} of $(\alpha_0,\beta_0)$ is defined by taking $\omega_i=\|B_i\|_2$ for $i=0:k$ in $\kappa_{\theta} ((\alpha_0,\beta_0),P)$ and is denoted by $\kappa_{\theta}^{r}((\alpha_0,\beta_0),P)$.
	\end{enumerate}
\end{definition}	
\smallskip
The absolute eigenvalue condition number in Definition \ref{weights} does not correspond to perturbations in the coefficients of $P$ appearing in applications, but it is studied because its analysis is the simplest one. Quoting Nick Higham \cite[p. 56]{higham-function-book}, ``it is the relative condition number that is of interest, but it is more convenient to state results for the absolute condition number''. The relative with respect to the norm of $P$ eigenvalue condition number corresponds to perturbations in the coefficients of $P$ coming from the backward errors of solving PEPs by applying a backward stable generalized eigenvalue algorithm to any reasonable linearization of $P$ \cite{block-kron,VD-DW}. Observe that $\kappa_{\theta}^{p}((\alpha_0,\beta_0),P) = \|P\|_\infty \, \kappa_{\theta}^{a}((\alpha_0,\beta_0),P)$ and, therefore, one of these condition numbers can be easily computed from the other. Finally, the relative eigenvalue condition number corresponds to perturbations in the coefficients of $P$ coming from an ``ideal'' coefficientwise backward stable algorithm for the PEP. Unfortunately, nowadays, such ``ideal'' algorithm exists only for degrees $k=1$ (the QZ algorithm for generalized eigenvalue problems) and $k=2$, in this case via linearizations and delicate scalings of $P$ \cite{FanLinVD,hamarling,zeng-su}. The recent work \cite{VanBarel-Tisseur} shows that there is still some hope of finding an ``ideal'' algorithm for PEPs with degree $k >2$.

In this paper, we will also compare the backward errors of approximate right and left eigenpairs of  the M\"obius transform $M_A(P)$ of a homogeneous matrix polynomial $P$ with  the  backward errors of approximate right and left eigenpairs of $P$ constructed from those of $M_A(P)$. Next we introduce the definition of backward errors of approximate eigenpairs of a homogeneous matrix polynomial.

\begin{definition}
	\label{def-backw}
	Let $(\widehat x,(\widehat{\alpha_0},\widehat{\beta_0}))$ be an approximate right eigenpair of the regular matrix polynomial $P(\alpha,\beta)=\sum_{i=0}^k\alpha^i\beta^{k-i}B_i$. We define the backward error of $(\widehat x,(\widehat{\alpha_0},\widehat{\beta_0}))$ as
	\begin{equation*}
	\eta_P(\widehat x,(\widehat{\alpha_0},\widehat{\beta_0})):=\min\{\epsilon: (P(\widehat{\alpha_0},\widehat{\beta_0})+\Delta P(\widehat{\alpha_0},\widehat{\beta_0}))\widehat x=0, \|\Delta B_i\|_2\leq\epsilon\; \omega_i, i=0:k\},
	\end{equation*}
	where $\Delta P(\alpha,\beta)=\sum_{i=0}^k\alpha^i\beta^{k-i}\Delta B_i$ and $\omega_i, i=0:k$, are nonnegative weights that allow flexibility in how the perturbations of $P(\alpha,\beta)$ are measured. Similarly, for an approximate left eigenpair $(\widehat{y}^*, (\widehat{\alpha_0}, \widehat{\beta_0}))$, we define 
	\begin{equation*}
	\eta_P(\widehat{y}^*,(\widehat{\alpha_0},\widehat{\beta_0})) :=\min\{\epsilon: \widehat{y}^*(P(\widehat{\alpha_0},\widehat{\beta_0})+\Delta P(\widehat{\alpha_0},\widehat{\beta_0}))=0, \|\Delta B_i\|_2\leq\epsilon\; \omega_i, i=0:k\}.
	\end{equation*}
\end{definition}
Next we present an explicit formula to compute the backward error of an approximate eigenpair of a homogeneous matrix polynomial.

\begin{theorem}{\rm\cite{HigLiTis, Tis2000}}\label{back-def} 
	Let $(\widehat x,(\widehat{\alpha_0},\widehat{\beta_0}))$ and $(\widehat{y}^*, (\widehat{\alpha_0}, \widehat{\beta_0}))$ be, respectively,  an approximate right and an approximate left eigenpair of the regular matrix polynomial $P(\alpha,\beta)=\sum_{i=0}^k\alpha^i\beta^{k-i}B_i$. Then,
	\begin{enumerate}
		\item $\displaystyle \eta_P(\widehat x,(\widehat{\alpha_0},\widehat{\beta_0}))= \frac{\|P(\widehat{\alpha_0}, \widehat{\beta_0})\widehat x\|_2}{(\sum_{i=0}^k |\widehat{\alpha_0}|^i |\widehat{\beta_0}|^{k-i}\omega_i) \|\widehat x\|_2},$ and 
		\item $\displaystyle \eta_P(\widehat{y}^*,(\widehat{\alpha_0},\widehat{\beta_0}))= \frac{\|\widehat{y}^* P(\widehat{\alpha_0}, \widehat{\beta_0})\|_2}{(\sum_{i=0}^k |\widehat{\alpha_0}|^i |\widehat{\beta_0}|^{k-i}\omega_i) \|\widehat y\|_2}.$
	\end{enumerate}
\end{theorem}

As in the case of condition numbers, the weights in Definition \ref{def-backw} can be chosen in different ways. We will consider the same three choices as in Definition  \ref{weights}, which leads to the following definition.

\begin{definition} \label{def.weigths-back} With the same notation and assumptions as in Definition \ref{def-backw}:
	\begin{enumerate}
		\item The \emph{absolute backward errors} of $(\widehat x,(\widehat{\alpha_0},\widehat{\beta_0}))$ and  $(\widehat{y}^*,(\widehat{\alpha_0},\widehat{\beta_0}))$ are defined by taking  $\omega_i=1$ for $i=0:k$ in $\eta_P (\widehat x,(\widehat{\alpha_0},\widehat{\beta_0}))$ and $\eta_P (\widehat{y}^*,(\widehat{\alpha_0},\widehat{\beta_0}))$, and are denoted by $\eta_P^a (\widehat x,(\widehat{\alpha_0},\widehat{\beta_0}))$ and $\eta_P^a (\widehat{y}^*,(\widehat{\alpha_0},\widehat{\beta_0}))$.
		\item The \emph{relative with respect to the norm of $P$ backward errors} of $(\widehat x,(\widehat{\alpha_0},\widehat{\beta_0}))$ and  $(\widehat{y}^*,(\widehat{\alpha_0},\widehat{\beta_0}))$ are defined by taking  $\omega_i=\|P\|_\infty$ for $i=0:k$, and are denoted by $\eta_P^p (\widehat x,(\widehat{\alpha_0},\widehat{\beta_0}))$ and $\eta_P^p (\widehat{y}^*,(\widehat{\alpha_0},\widehat{\beta_0}))$.
		\item The \emph{relative backward errors} of $(\widehat x,(\widehat{\alpha_0},\widehat{\beta_0}))$ and  $(\widehat{y}^*,(\widehat{\alpha_0},\widehat{\beta_0}))$ are defined by taking $\omega_i=\|B_i\|_2$ for $i=0:k$, and are denoted by $\eta_P^r (\widehat x,(\widehat{\alpha_0},\widehat{\beta_0}))$ and $\eta_P^r (\widehat{y}^*,(\widehat{\alpha_0},\widehat{\beta_0}))$.
	\end{enumerate}
\end{definition}

\section{Effect of M\"{o}bius transformations on  eigenvalue condition numbers}
\label{SecHom} 
This section contains  the most important results of this paper (presented in Subsection \ref{subsec.52}), which are obtained from the key and technical Theorem \ref{quotient-expr-1} (included in Subsection \ref{subsec.51}). In Subsection \ref{subsec.53} we present some additional results.

Throughout this section, $P(\alpha, \beta)\in \mathbb{C}[\alpha,\beta]^{n\times n}_k$ is a regular homogeneous matrix polynomial and $(\alpha_0,\beta_0)$ is a simple eigenvalue of $P(\alpha, \beta)$. Moreover, $M_A$ is a  M\"{o}bius transformation on $\mathbb{C}[\alpha,\beta]^{n\times n}_k$ and $\langle A^{-1}[\alpha_0, \beta_0]^T\rangle$ is the eigenvalue of $M_A(P)$ associated with $(\alpha_0,\beta_0)$ introduced in Definition \ref{def.associated}. We are interested in studying the influence of the M\"{o}bius transformation $M_A$ on the Stewart-Sun eigenvalue condition number, that is, we would like to compare the Stewart-Sun  condition numbers of $(\alpha_0,\beta_0)$ and $\langle A^{-1}[\alpha_0, \beta_0]^T\rangle$. More precisely, our goal is to determine sufficient conditions on $A$, $P$ and $M_A(P)$ so that  the condition number of $\langle A^{-1}[\alpha_0, \beta_0]^T\rangle$ is similar to that of $(\alpha_0,\beta_0)$, \emph{independently of the particular eigenvalue $(\alpha_0,\beta_0)$ that is considered}. With this goal in mind, we first obtain  upper and lower bounds on the quotient
\begin{equation}\label{notation-hom}
Q_{\theta}:=\frac{\kappa_{\theta}(\langle A^{-1}[\alpha_0, \beta_0]^T\rangle, M_A(P))}{\kappa_{\theta}((\alpha_0, \beta_0),P)}
\end{equation}
which are independent of $(\alpha_0, \beta_0)$ 
and, then, we find conditions that make these upper and lower bounds approximately equal to one or, more precisely, moderate numbers. 

In view of Definition \ref{weights}, three variants of the quotient \eqref{notation-hom}, denoted by $Q_{\theta}^a, Q_{\theta}^p,$ and $Q_{\theta}^r$, are considered, which correspond, respectively, to quotients of absolute, relative with respect to the norm of the polynomial, and relative eigenvalue condition numbers. \emph{The lower and upper bounds for $Q_{\theta}^a$ and $Q_{\theta}^p$ are presented in Theorems \ref{main-homo0} and \ref{main-homo1} and depend only on $A$ and the degree $k$ of $P$}. So, these bounds lead to very simple sufficient conditions, valid for \emph{all} polynomials and simple  eigenvalues, that allow us to identify some M\"{o}bius transformations which do not change significantly the condition numbers.  \emph{The lower and upper bounds for $Q_{\theta}^r$ are presented in Theorem \ref{main-homo} and depend only on $A$,  the degree $k$ of $P$, and some ratios of the norms of the matrix coefficients of $P$ and $M_A(P)$}. These bounds also lead to simple sufficient conditions, valid for \emph{all} simple eigenvalues but only for certain matrix polynomials, that allow us to identify some M\"{o}bius transformations which do not change significantly the condition numbers.

The first obstacle we have found in obtaining the results described in the previous paragraph is that a direct application of Theorem \ref{teorhomcondnumb} leads to a very complicated expression for the quotient $Q_{\theta}$ in \eqref{notation-hom}. Therefore, in Theorem \ref{quotient-expr-1} we deduce an expression for $Q_{\theta}$ that depends only on $(\alpha_0,\beta_0)$, the matrix $A$ inducing the M\"{o}bius transformation, and the weights  
$\widetilde{\omega}_i$ and $\omega_i$ used in $\kappa_{\theta}(\langle A^{-1}[\alpha_0, \beta_0]^T\rangle, M_A(P))$ and $\kappa_{\theta}((\alpha_0, \beta_0),P)$, respectively. Thus, this expression gets rid of  the partial derivatives of $P$ and $M_A(P)$.

\subsection{A derivative-free expression for the quotient of condition numbers} \label{subsec.51} The derivative-free expression for $Q_{\theta}$  obtained in this section is \eqref{Qrah}. Before diving into the details of its proof, we emphasize that, even though the formula for the Stewart-Sun condition number is  independent of the representative of the eigenvalue, the expression  \eqref{Qrah} is  independent of the particular representative $[\alpha_0 , \beta_0]^T$ chosen for the eigenvalue $(\alpha_0 , \beta_0)$ of $P$ but \emph{not} of the representative of the associated eigenvalue  of $M_A (P)$, which must be $A^{-1}[\alpha_0, \beta_0]^T$. 
A second remarkable feature of \eqref{Qrah} is that  it depends on the determinant of the matrix $A$ inducing the M\"{o}bius transformation. Note also that $\mathrm{det} (A)$ cannot be removed by choosing a different representative of $\langle A^{-1}[\alpha_0, \beta_0]^T \rangle$. 


\begin{theorem}\label{quotient-expr-1}
	Let $P(\alpha,\beta)= \sum_{i=0}^k \alpha^i \beta^{k-i} B_i \, \in \mathbb{C}[\alpha,\beta]^{n\times n}_k$ be a regular homogeneous matrix polynomial and let $A=\left[\begin{array}{cc} a & b \\ c & d\end{array} \right] \in GL(2, \mathbb{C})$. Let $M_A(P)(\gamma, \delta)= \sum_{i=0}^k \gamma^i \delta^{k-i}  \widetilde{B}_i  \in \mathbb{C}[\alpha,\beta]^{n\times n}_k$ be the M\"obius transform of $P(\alpha, \beta)$ under $M_A$. Let $(\alpha_0, \beta_0)$ be a simple eigenvalue of $P(\alpha, \beta)$ and let $[\alpha_0, \beta_0]^T$ be a representative of $(\alpha_0, \beta_0)$. Let $[\gamma_0, \delta_0]^T := A^{-1}[\alpha_0, \beta_0]^T$ be the representative of the eigenvalue of $M_A(P)$ associated with $[\alpha_0, \beta_0]^T$.
	Let $Q_{\theta}$ be as in (\ref{notation-hom}) and let $\omega_i$ and $\widetilde{w}_i$ be the weights in the definition of the Stewart-Sun eigenvalue condition number associated with  the eigenvalues $(\alpha_0,\beta_0)$ and $\langle A^{-1}[\alpha_0,\beta_0]^T \rangle$, respectively. Then, \begin{equation}\label{Qrah} Q_{\theta} =
	\frac{\sum_{i=0}^k \left | \gamma_0 \right |^{i} \left |\delta_0 \right |^{(k-i)}\widetilde{\omega}_i }{|\mathrm{det}(A)| \, \sum_{i=0}^k |\alpha_0|^{i}|\beta_0|^{(k-i)}\omega_i}
	\frac{|\alpha_0|^2 + |\beta_0|^2}{\left |\gamma_0 \right |^2 + \left |\delta_0 \right|^2}.
	\end{equation}
	Moreover, (\ref{Qrah}) is independent of the choice of representative for $(\alpha_0, \beta_0)$.
\end{theorem}

\begin{proof} In order to prove the formula (\ref{Qrah}), we compute $\kappa_{\theta}(\langle A^{-1}[\alpha_0,\beta_0]^T\rangle, M_A(P))$ and $\kappa_{\theta}((\alpha_0,\beta_0),P)$ separately,  and then calculate their quotient. Since the definition of the Stewart-Sun eigenvalue condition number is independent of the choice of representative of the eigenvalue, when computing the condition numbers of $(\alpha_0, \beta_0)$ and $\langle A^{-1}[\alpha_0,\beta_0]^T\rangle$,  we have freedom to choose any representative. In this proof, we choose an arbitrary representative  $[\alpha_0,\beta_0]^T$  of $(\alpha_0,\beta_0)$ and, once $[\alpha_0,\beta_0]^T$ is fixed, we choose  $[\gamma_0, \delta_0]^T:=A^{-1}[\alpha_0, \beta_0]^T$ as the representative of the eigenvalue of $M_A(P)$ associated with $(\alpha_0, \beta_0)$.
	
	We first compute $\kappa_{\theta}((\alpha_0,\beta_0),P)$. Let $x$ and $y$ be, respectively, a right and a left eigenvector of $P(\alpha,\beta)$ associated with $(\alpha_0, \beta_0)$. We start by simplifying the denominator of \eqref{hom-form}. Note that { 
		\begin{align}\label{partialP}
		D_{\alpha} P(\alpha, \beta) &= \sum_{i=1}^k i \alpha^{i-1} \beta^{k-i} B_i, 
		\qquad \mathrm{and}  \\
		\label{partialP2}
		D_{\beta} P(\alpha, \beta) & = \sum_{i=0}^{k-1} (k-i) \alpha^{i} \beta^{k-i-1} B_i= \sum_{i=0}^{k} (k-i) \alpha^{i} \beta^{k-i-1} B_i.
		\end{align}}

	We consider two cases.
	
	Case I: Assume that $\beta_0\neq 0$. Evaluating (\ref{partialP}) and (\ref{partialP2})  at $[\alpha_0, \beta_0]^T$, we get
	\begin{align*}
	&\overline{\beta_0} D_{\alpha} P(\alpha_0, \beta_0) - \overline{\alpha_0} D_{\beta} P(\alpha_0, \beta_0) \\
	&= |\beta_0|^2 \sum_{i=1}^{k} i \alpha_0^{i-1} \beta_0^{k-i-1} B_i -\overline{\alpha_0} \sum_{i=0}^{k} (k-i) \alpha_0^{i} \beta_0^{k-i-1} B_i \\
	&=( |\beta_0|^2 + |\alpha_0|^2) \sum_{i=1}^{k} i \alpha_0^{i-1} \beta_0^{k-i-1}B_i - \overline{\alpha_0} k \sum_{i=0}^{k} \alpha_0^{i} \beta_0^{k-i-1} B_i.
	\end{align*}
	Moreover, { 
		\begin{align}\label{P1}
		&| y^*(\overline{\beta_0} D_{\alpha} P(\alpha_0, \beta_0) - \overline{\alpha_0} D_{\beta} P(\alpha_0, \beta_0)) x| \nonumber \\
		&=\left| y^* \left (( |\beta_0|^2 + |\alpha_0|^2) \sum_{i=1}^{k} i \alpha_0^{i-1} \beta_0^{k-i-1}B_i - \frac{\overline{\alpha_0} k}{\beta_0} \sum_{i=0}^{k} \alpha_0^{i} \beta_0^{k-i} B_i\right) x \right | \nonumber\\
		&= ( |\beta_0|^2 + |\alpha_0|^2) \left|y^* \left (  \sum_{i=1}^{k} i \alpha_0^{i-1} \beta_0^{k-i-1}B_i  \right) x \right |,
		\end{align}
		where the last equality follows from $P(\alpha_0, \beta_0) x = 0$. }
	Thus, if $\beta_0\neq 0$, 
	\begin{equation}\label{P-no)} \kappa_{\theta}((\alpha_0,\beta_0),P)= \frac{\left (\sum_{i=0}^k|\alpha_0|^{i}|\beta_0|^{(k-i)} \omega_i \right )\|y\|_2\|x\|_2}{( |\beta_0|^2 + |\alpha_0|^2) \left|y^* \left (  \sum_{i=1}^{k} i \alpha_0^{i-1} \beta_0^{k-i-1}B_i  \right) x \right |}.
	\end{equation}
	
	Case II: If $\beta_0=0$, evaluating (\ref{partialP2})  at $[\alpha_0, \beta_0]^T$, we get that the denominator of \eqref{hom-form} is $|\alpha_0|^k |y^* B_{k-1} x|$. Thus,
	\begin{equation}\label{P-0} \kappa_{\theta}((\alpha_0,\beta_0),P)=\left (\sum_{i=0}^k|\alpha_0|^{i}|\beta_0|^{(k-i)} \omega_i \right) \frac{\|y\|_2\|x\|_2}{|\alpha_0|^k |y^* B_{k-1} x|}.
	\end{equation}

	
	\medskip	
	Next, we compute $\kappa_{\theta}(\langle [\gamma_0,\delta_0]^T\rangle,M_A(P))$ and express it in terms of the coefficients of $P$. As above, we start by simplifying the denominator of \eqref{hom-form} when $P(\alpha, \beta)$ is replaced by $M_A(P)(\gamma, \delta)$ and $[\alpha_0, \beta_0]^T$ is replaced by $[\gamma_0,\delta_0]^T$. Recall that, by Lemma \ref{eig-hom},  $x$ and $y$ are, respectively, a right and a left eigenvector of $M_A(P)$ associated with $\langle [\gamma_0, \delta_0]^T \rangle$.  Note that, since
	$M_A(P)(\gamma, \delta)= \sum_{i=0}^k (a \gamma + b\delta)^i (c \gamma + d \delta)^{k-i} B_i,$ we have
	\begin{align}
	D_{\gamma}M_A(P)(\gamma, \delta) 
	&=\sum_{i=1}^k a i (a\gamma+b\delta)^{i-1}(c\gamma+d\delta)^{k-i}B_i \nonumber \\ & \phantom{=} + \sum_{i=0}^{k} c (k-i) (a\gamma+b\delta)^i (c\gamma + d \delta)^{k-i-1} B_i , \label{partialM1} \\
	D_{\delta}M_A(P)(\gamma, \delta)
	&= \sum_{i=1}^k b i (a\gamma+b\delta)^{i-1}(c\gamma+d\delta)^{k-i}B_i \nonumber
	\\ & \phantom{=} + \sum_{i=0}^{k} d (k-i) (a\gamma+b\delta)^i (c\gamma + d \delta)^{k-i-1} B_i. \label{partialM2}
	\end{align}

	Again, we consider two cases.
	
	Case I: Assume that $\beta_0 \neq 0$.  We evaluate (\ref{partialM1}) and (\ref{partialM2}) at $ [\gamma_0,\delta_0]^T=\left [\frac{d\alpha_0-b\beta_0}{\mathrm{det}(A)},\frac{a\beta_0-c\alpha_0}{\mathrm{det}(A)}\right]$, 
	and get
	\begin{align*}
	D_{\gamma} M_A(P)(\gamma_0, \delta_0) &=\left[a \sum_{i=1}^k  i \alpha_0^{i-1} \beta_0 ^{k-i} B_i +c  \sum_{i=0}^{k}  (k-i) \alpha_0^i \beta_0^{k-i-1} B_i \right]\\
	&= \left [(a \beta_0-c\alpha_0) \sum_{i=1}^{k}  i \alpha_0^{i-1} \beta_0^{k-i-1} B_i + c k \sum_{i=0}^{k}\alpha_0^{i} \beta_0^{k-i-1} B_i\right]\\
	&= \left [\mathrm{det}(A) \delta_0 \sum_{i=1}^{k}  i \alpha_0^{i-1} \beta_0^{k-i-1} B_i + \frac{c k}{\beta_0} P(\alpha_0, \beta_0)\right].
	\end{align*}
	An analogous computation shows that
	\begin{align*}
	D_{\delta} M_A(P)(\gamma_0, \delta_0)
	&=\left [-\mathrm{det}(A) \gamma_0 \sum_{i=1}^{k}  i \alpha_0^{i-1} \beta_0^{k-i-1} B_i + \frac{d k}{\beta_0}P(\alpha_0, \beta_0)\right].
	\end{align*}
	This implies that,  
	\begin{align}\label{M1}
	&| y^* (\overline{\delta_0} D_{\gamma}M_A(P)(\gamma_0, \delta_0) - \overline{\gamma_0} D_{\delta}M_A(P)(\gamma_0, \delta_0))x|\nonumber\\
	&\phantom{aaaaaaa} =|\mathrm{det}(A)|(|\delta_0|^2+|\gamma_0|^2) \left |y^*\left (\sum_{i=1}^{k} i \alpha_0^{i-1} \beta_0^{k-i-1}B_i\right )x \right |.
	\end{align}
	Thus, if $\beta_0\neq 0$, 
	\begin{equation}\label{MAP-no0} \kappa_{\theta}(\langle [\gamma_0, \delta_0]^T \rangle,M_A(P) )= \frac{\left (\sum_{i=0}^k|\gamma_0|^{i}|\delta_0|^{(k-i)} \widetilde{\omega_i} \right)\|y\|_2\|x\|_2}{|\mathrm{det}(A)|( |\gamma_0|^2 + |\delta_0|^2) \left|y^* \left (  \sum_{i=1}^{k} i \alpha_0^{i-1} \beta_0^{k-i-1}B_i  \right) x \right |}.
	\end{equation}

	Case II: If $\beta_0=0$, since $A[\gamma_0, \delta_0]^T=[\alpha_0, \beta_0]^T$, we deduce that $c\gamma_0+d \delta_0 =0$. Moreover, by \eqref{rel-eig},  $\gamma_0= d\alpha_0/\mathrm{det}(A)$ and $\delta_0=-c\alpha_0/\mathrm{det}(A)$. Since $x$ is a right eigenvector of $P(\alpha, \beta)$ with eigenvalue $(\alpha_0,0)$, we have that $0=P(\alpha_0, 0) x = \alpha_0^k B_k x$ which implies $B_k x =0$ since $\alpha_0 \neq 0$. Using all this information and some algebraic manipulations, we get that, if $\beta_0 = 0$,
	\begin{equation}\label{MAP-0}  \kappa_{\theta}(\langle [\gamma_0,\delta_0]^T\rangle,M_A (P))= \frac{\left (\sum_{i=0}^k|\gamma_0|^{i}|\delta_0|^{(k-i)} \widetilde{\omega_i} \right)\|y\|_2\|x\|_2}{|\det(A)|( |\gamma_0|^2 + |\delta_0|^2) |\alpha_0|^{k-2} |y^* B_{k-1} x|}.
	\end{equation}
	
	Finally we compute $Q_{\theta}$. Note that, from (\ref{P-no)}), (\ref{P-0}), (\ref{MAP-no0}) and (\ref{MAP-0}), we get (\ref{Qrah}), regardless of the value of $\beta_0$. Moreover, note that \eqref{Qrah} does not change if $[\alpha_0, \beta_0]^T$ is replaced by $[t \, \alpha_0, t\, \beta_0]^T$ for any complex number $t\ne 0$.
\end{proof}


In the spirit of Definition \ref{weights}, when comparing the condition number of an eigenvalue $(\alpha_0, \beta_0)$ of $P$ and the associated eigenvalue of $M_A(P)$, we will consider the three quotients introduced in the next definition.
\begin{definition} \label{notation-hom3}
	With the same notation and assumptions as in Theorem \ref{quotient-expr-1}, we define the following three quotients of condition numbers:
	\begin{enumerate}
		\item $\displaystyle Q_{\theta}^{a}:=\frac{\kappa_{\theta}^{a}(\langle A^{-1}[\alpha_0, \beta_0]^T\rangle, M_A(P)) }{\kappa_{\theta}^{a}((\alpha_0, \beta_0),P)}$, which is called the absolute quotient.
		\item $\displaystyle Q_{\theta}^{p}:=\frac{\kappa_{\theta}^{p}(\langle A^{-1}[\alpha_0, \beta_0]^T\rangle, M_A(P)) }{\kappa_{\theta}^{p}((\alpha_0, \beta_0),P)}$, which is called the relative quotient with respect to the norms of $M_A(P)$ and $P$.
		\item $\displaystyle Q_{\theta}^{r}:=\frac{\kappa_{\theta}^{r}(\langle A^{-1}[\alpha_0, \beta_0]^T\rangle, M_A(P)) }{\kappa_{\theta}^{r}((\alpha_0, \beta_0),P)}$, which is called the relative quotient.
	\end{enumerate}
\end{definition}
\medskip
Combining Definition \ref{weights} and the expression \eqref{Qrah}, we obtain immediately expressions for $Q_{\theta}^{a}$, $Q_{\theta}^{p}$, and $Q_{\theta}^{r}$ as explained in the following corollary.
\begin{corollary} \label{quotient-expr-2} With the same notation and assumptions as in Theorem \ref{quotient-expr-1}:
	\begin{enumerate}
		\item $Q_{\theta}^{a}$ is obtained from \eqref{Qrah} by taking  $\omega_i=\widetilde{\omega}_i=1$ for $i=0:k$.
		\item $Q_{\theta}^{p}$ is obtained from \eqref{Qrah} by taking  $\omega_i=\displaystyle{\max_{j=0:k}}\{\|B_j\|_2\}$ and $\widetilde{\omega}_i=\displaystyle{\max_{j=0:k}}\{\|\widetilde B_j\|_2\}$ for $i=0:k$.
		\item  $Q_{\theta}^{r}$ is obtained from \eqref{Qrah} by taking $\omega_i=\|B_i\|_2$ and  $\widetilde{\omega}_i=\|\widetilde B_i\|_2$ for $ i=0:k$.
	\end{enumerate}
\end{corollary}





\subsection{Eigenvalue-free bounds on the quotients of condition numbers} \label{subsec.52}
{The first goal of this section is to find lower and upper bounds on the quotients $Q^a_{\theta}$, $Q^p_{\theta}$, and $Q^r_{\theta}$, introduced in Definition \ref{notation-hom3}, that are independent of the considered eigenvalues. The second goal is to provide simple sufficient conditions guaranteeing that the obtained bounds are moderate numbers, i.e., not far from one. The bounds on $Q^a_{\theta}$ are obtained from the expression \eqref{Qrah} in Theorem \ref{quotient-expr-1}. The proofs of the bounds on $Q^p_{\theta}$ and $Q^r_{\theta}$ also require \eqref{Qrah}, but, in addition, Proposition \ref{cotaMobius} is used.

	The bounds on $Q_{\theta}^p$ and $Q_{\theta}^r$ can be expressed in terms of the condition number of the matrix  $A \in GL(2, \mathbb{C})$ that induces the M\"obius transformation. We will  use the infinite condition number of $A$, that is,
	$$\mathrm{cond}_{\infty}(A) := \|A\|_{\infty} \|A^{-1}\|_{\infty}.$$
	It is interesting to highlight that the bounds on the quotients $Q^a_{\theta}$, $Q^p_{\theta}$, and $Q^r_{\theta}$ in Theorems \ref{main-homo0}, \ref{main-homo1}, and \ref{main-homo} will require different types of proofs  for polynomials of degree $k=1$ and for polynomials of degree $k \geq 2$. In fact, this is related to actual differences in the behaviours of these quotients for  polynomials of degree $1$ and larger than $1$ when the matrix $A$ inducing the M\"{o}bius transformation is ill-conditioned. These questions are studied in Subsection \ref{subsec.53}.
	
	The next theorem presents the announced upper and lower bounds on $Q_{\theta}^{a}$.
	\begin{theorem}\label{main-homo0}
		Let $P(\alpha,\beta) \in \mathbb{C}[\alpha,\beta]_k^{n\times n}$ be a regular homogeneous matrix polynomial and let $A \in GL(2, \mathbb{C})$.  Let $(\alpha_0, \beta_0)$ be a simple eigenvalue of $P(\alpha, \beta)$ and let $\langle A^{-1}[\alpha_0, \beta_0]^T\rangle$ be the eigenvalue of $M_A(P) (\gamma, \delta)$ associated with $(\alpha_0, \beta_0)$. 
		Let $Q_{\theta}^{a}$ be the absolute quotient in Definition \ref{notation-hom3}(1.) and let $S_k:=4(k+1)$. 
		\begin{enumerate}
			\item If $k=1$, then
			$$\frac{1}{2 \|A\|_{\infty}} \leq Q_{\theta}^{a}\leq 2 \|A^{-1}\|_{\infty}.$$
			\item If $k\geq 2$, then
			$$
			\frac{\|A^{-1}\|_{\infty}}{S_k \,  \|A\|_{\infty}^{k-1}}\leq
			Q_{\theta}^{a} \leq S_k \, \frac{ \|A^{-1}\|_{\infty}^{k-1}}{\|A\|_{\infty} }.
			$$	
		\end{enumerate}
	\end{theorem}
	
	\begin{proof}
		Let $A =\left[ \begin{array}{cc} a & b \\c & d \end{array} \right]$.	As in the proof of Theorem \ref{quotient-expr-1}, we choose  an arbitrary representative $[\alpha_0 , \beta_0]^T$ of $(\alpha_0 , \beta_0)$, and the associated representative $[\gamma_0 , \delta_0]^T :=A^{-1}[\alpha_0, \beta_0]^T=\left [\frac{ d\alpha_0-b\beta_0}{\mathrm{det}(A)}, \frac{a\beta_0-c\alpha_0}{\mathrm{det}(A)}\right]$ of the eigenvalue of $M_A(P)$.  We obtain first the upper bounds.
		
		If $k=1$, then, from Corollary \ref{quotient-expr-2}(1.) and  \eqref{quotient-norm-eig}, and recalling that $\frac{1}{\sqrt{2}} \|x\|_1 \leq \|x\|_2\leq \|x\|_1$ for every $2\times 1$ vector $x,$ we get
		\begin{align}
		Q_{\theta}^{a} &= \frac{1}{|\mathrm{det}(A)|}\frac{(|\gamma_0| +|\delta_0|)(|\alpha_0|^2 +|\beta_0|^2)}{(|\alpha_0|+|\beta_0|)(|\gamma_0|^2+|\delta_0|^2)} = \frac{1}{|\mathrm{det}(A)|} \frac{\|[\gamma_0, \delta_0]^T\|_1\|[\alpha_0, \beta_0]^T\|_2^2}{\|[\alpha_0, \beta_0]^T\|_1\|[\gamma_0, \delta_0]^T\|_2^2}\label{Qa-k1-1}\\
		& \leq \frac{2}{|\mathrm{det}(A)|} \frac{\|[\alpha_0, \beta_0]^T\|_1}{\|[\gamma_0, \delta_0]^T\|_1}\leq \frac{2\|A\|_1}{|\det(A)|}=2\|A^{-1}\|_{\infty}. \label{Qa-k1-2}
		\end{align}
		
		
		
		If $k \geq 2$, then by using again Corollary \ref{quotient-expr-2}(1.) and \eqref{quotient-norm-eig}, and recalling that $ \|x\|_{\infty}\leq \|x\|_2\leq \sqrt{2}\|x\|_{\infty}$ for every $2\times 1$ vector $x,$ we have
		\begin{align} 
		Q_{\theta}^{a} & \leq \frac{(k+1)\ }{|\det(A)|}\frac{\max\{|\gamma_0|^k, |\delta_0|^k\}}{(|\alpha_0|^{k}+|\beta_0|^{k})}\frac{ \|[\alpha_0, \beta_0|^T\|_2^2}{\|[\gamma_0, \delta_0]^T\|_2^2}\label{casek2-2} \\
		& \leq  \frac{2(k+1)\ }{|\det(A)|} \frac{ \| [\gamma_0, \delta_0]^T\|_{\infty}^{k-2}}{\| [\alpha_0, \beta_0]^T \|_{\infty}^{k-2}} \label{casek2-3} \\
		& \leq 2(k+1) \, \frac{ \|A^{-1}\|_{\infty}^{k-2}}{|\mathrm{det}(A)| }= 2(k+1) \, \frac{ \|A^{-1}\|_{\infty}^{k-1}}{\|A\|_1 },\label{casek2-4}
		\end{align}
		and the upper bound for $Q_{\theta}^a$ follows taking into account that $\|A\|_1 \geq \frac{\|A\|_{\infty}}{2}.$	 	
		To obtain the lower bounds, note that Proposition \ref{prop-mobius}(2.) implies
		\begin{equation} \label{eq.invlowbounds}
		\frac{1}{Q^a_{\theta}} =	\frac{\kappa_{\theta}^{a}((\alpha_0, \beta_0),P)}{\kappa_{\theta}^{a}(\langle A^{-1}[\alpha_0, \beta_0]^T\rangle, M_A(P)) }=
		\frac{\kappa_{\theta}^{a}((\alpha_0, \beta_0),M_{A^{-1}}(M_A(P)))}{\kappa_{\theta}^{a}(\langle A^{-1}[\alpha_0, \beta_0]^T\rangle, M_A(P)) }.
		\end{equation}
		The previously obtained upper bounds can be applied to the right-most quotient in \eqref{eq.invlowbounds} with $A$ and $A^{-1}$ interchanged. This leads to the lower bounds for $Q^a_{\theta}$.
	\end{proof}

	\begin{remark} (Discussion on the bounds in Theorem \ref{main-homo0}) \label{rem.absolute1} 
		\begin{equation} \label{eq.suffabs}
		\|A^{-1}\|_{\infty}\approx 1\quad \textrm{ and} \quad  \|A\|_{\infty} \approx 1
		\end{equation}
		are sufficient to imply that all the bounds in Theorem \ref{main-homo0} are moderate numbers since the factor in the bounds depending on $k$ is small for moderate $k$. Therefore, the conditions \eqref{eq.suffabs}, which involve only $A$, guarantee that the M\"{o}bius transformation $M_A$ \emph{does not change significantly the absolute eigenvalue condition number of any simple eigenvalue  of any matrix polynomial}.
		Observe that \eqref{eq.suffabs} implies, in particular, that  $\mathrm{cond}_{\infty}(A) \approx 1$, although the reverse implication does not hold. 
		
		For $k=1$ and $k>2$, the conditions \eqref{eq.suffabs} are also necessary for the bounds in Theorem \ref{main-homo0} to be moderate numbers. This is obvious for $k=1$. For $k >2$, note that $\|A^{-1}\|_{\infty}/\|A\|_{\infty}^{k-1}\approx 1$ and $\|A^{-1}\|_{\infty}^{k-1}/\|A\|_{\infty} \approx 1$ imply $\|A\|_{\infty}^{k^2 -2k} \approx 1$ and $\|A^{-1}\|_{\infty}^{k^2 -2k} \approx 1$, and, thus,
		$\|A\|_{\infty} \approx 1$ and $\|A^{-1}\|_{\infty} \approx 1$.
		
		However, the quadratic case $k=2$ is different because  the bounds in Theorem \ref{main-homo0} can be moderate in cases in which  the conditions \eqref{eq.suffabs} are not satisfied. For $k=2$, the lower and upper bounds are  $\|A^{-1}\|_{\infty}/(12 \, \|A\|_{\infty})$ and $12 \, \|A^{-1}\|_{\infty}/\|A\|_{\infty}$, which are moderate under the unique necessary and sufficient condition $\|A^{-1}\|_{\infty} \approx \|A\|_{\infty}$.
		
		Notice that the very important  Cayley transformations introduced in Definition \ref{cayley} satisfy $\|A\|_{\infty}=2$ and $\|A^{-1}\|_{\infty} =1$ and, so, they satisfy \eqref{eq.suffabs}. The same happens for the reversal M\"obius transformation in Example \ref{ex-reversal} since $\|R\|_{\infty} = \|R^{-1}\|_{\infty} =1$. 
	\end{remark}
	\smallskip
	
	The  next theorem presents the bounds on $Q_{\theta}^{p}$. As explained in the proof, these bounds can be readily obtained from combining Theorem \ref{main-homo0} and Proposition \ref{cotaMobius}. 
	\begin{theorem} \label{main-homo1}
		Let $P(\alpha,\beta) \in \mathbb{C}[\alpha,\beta]_k^{n\times n}$ be a regular homogeneous matrix polynomial and let $A \in GL(2, \mathbb{C})$.  Let $(\alpha_0, \beta_0)$ be a simple eigenvalue of $P(\alpha, \beta)$ and let $\langle A^{-1}[\alpha_0, \beta_0]^T \rangle$ be the eigenvalue of $M_A(P) (\gamma, \delta)$ associated with $(\alpha_0, \beta_0)$. 
		Let $Q_{\theta}^{p}$ be the relative quotient with respect to the norms of $M_A(P)$ and $P$ in Definition \ref{notation-hom3}(2.) and let $Z_k:=4(k+1)^{2}{k \choose \lfloor k/2\rfloor}$. 
		\begin{enumerate}
			\item If $k=1$, then
			$$\frac{1}{4 \; \mathrm{cond}_{\infty}(A)}\leq  Q_{\theta}^{p}\leq  4\; \mathrm{cond}_{\infty}(A).$$
			\item If $k\geq 2$, then
			$$
			\frac{1}{Z_k \, \mathrm{cond}_{\infty}(A)^{k-1}}\leq 
			Q_{\theta}^{p} \leq  Z_k\, \mathrm{cond}_{\infty}(A)^{k-1}.
			$$
		\end{enumerate}
	\end{theorem}

	\begin{proof} We only prove the upper bounds, since the lower bounds can be obtained from the upper bounds using an argument similar to the one used in \eqref{eq.invlowbounds}.
		Notice that parts (1.) and (2.) in Corollary \ref{quotient-expr-2} and Proposition \ref{cotaMobius} imply
		\begin{equation} \label{eq.factatop}
		Q_{\theta}^{p}=Q_{\theta}^{a} \, \frac{\max \limits_{i=0:k}\{\|\widetilde B_i\|_2\}}{\max \limits_{i=0:k}\{\|B_i\|_2\}} \leq Q_{\theta}^{a} \,(k+1){k \choose \lfloor k/2 \rfloor}\|A\|_{\infty}^k.
		\end{equation}
		Now, the upper bounds follow from the upper bounds on $Q_{\theta}^{a}$ in Theorem \ref{main-homo0}. \end{proof}
	
	
	\begin{remark}\label{rem-hom-1} (Discussion on the bounds in Theorem \ref{main-homo1}) The first observation on the bounds presented in Theorem \ref{main-homo1} is that the factor $Z_k$, depending only on the degree $k$ of $P$, becomes very large even for moderate values of $k$ (consider, for instance, $k=15$). This fact makes the lower and upper bounds very different from each other, even for matrices $A$ whose condition number is  close to 1, and, so, Theorem \ref{main-homo1} is useless for large $k$. However, we will see in the numerical experiments in Section \ref{sec:num} that the factor $Z_k$ is very pessimistic, i.e., although there is an observable dependence of the true values of the quotient $Q^p_{\theta}$ (and also of  $Q^r_{\theta}$) on $k$, such dependence is much smaller than the one predicted by $Z_k$. Moreover, in many important applications of matrix polynomials, $k$ is very small and so is $Z_k$ \cite{NLEVP-2013}. For instance, the linear case $k=1$ (generalized eigenvalue problem) and the quadratic case $k=2$ (quadratic eigenvalue problem) are particularly important. Therefore, in the informal discussion in this remark the factor $Z_k$ is ignored, but the reader should bear in mind that the obtained conclusions only hold rigorously for small degrees $k$.
		
		In our opinion, Theorem \ref{main-homo1} is the most illuminating result in this paper because it refers to the comparison of condition numbers that are very interesting in numerical applications (recall the comments in the paragraph just after Definition \ref{weights}) and also because it delivers a very clear sufficient condition that guarantees that the M\"{o}bius transformation $M_A$ \emph{does not change significantly the relative-with-respect-to-the-norm-of-the-polynomial eigenvalue condition number of any simple eigenvalue  of any matrix polynomial}. This sufficient condition is simply that the matrix $A$ is well-conditioned, since $\mathrm{cond}_{\infty}(A) \approx 1$ if and only if the lower and upper bounds in Theorem \ref{main-homo1} are moderate numbers. Notice that the very important  Cayley transformations in Definition \ref{cayley} satisfy $\mathrm{cond}_{\infty}(A) =2$ and that the reversal M\"obius transformation in Example \ref{ex-reversal} satisfies $\mathrm{cond}_{\infty}(R) =1$.
	\end{remark}
	\smallskip
	
	In the last part of this subsection, we present and discuss the bounds on $Q_{\theta}^r$. As previously announced, these bounds depend on $A$, $P$, and $M_A(P)$ and, so, are qualitatively different from the bounds on $Q_{\theta}^a$ and  $Q_{\theta}^p$ presented in Theorems \ref{main-homo0} and \ref{main-homo1}, which only depend on $A$ and the degree $k$ of $P$. In order to simplify the bounds, we will assume that the matrix coefficients with indices 0 and $k$   of $P$  and $M_A(P)$ (i.e. $B_0$,  $B_k$, $\widetilde{B}_0$ and $\widetilde{B}_k$) are different from zero, which covers the most interesting cases in applications.
	
	\begin{theorem}\label{main-homo}
		Let $P(\alpha,\beta)=\sum_{i=0}^k \alpha^i \beta^{k-i}B_i \in \mathbb{C}[\alpha , \beta]^{n \times n}_k$ be a regular homogeneous matrix polynomial and let $A \in GL(2, \mathbb{C})$. Let $(\alpha_0, \beta_0)$ be a simple eigenvalue of $P(\alpha,\beta)$ and let $\langle A^{-1}[\alpha_0, \beta_0]^T \rangle$ be the eigenvalue of $M_A(P) (\gamma,\delta)=\sum_{i=0}^k \gamma^i \delta^{k-i} \widetilde{B}_i$ associated with $(\alpha_0, \beta_0)$. Let $Q_{\theta}^{r}$ be the relative quotient in Definition \ref{notation-hom3}(3.) and let $Z_k:=4(k+1)^{2}{k \choose \lfloor k/2\rfloor}$. Assume that $B_0 \ne 0, B_k \ne 0, \widetilde{B}_0 \ne 0$, and $\widetilde{B}_k \ne 0$ and define
		\begin{equation}\label{rho} \rho:= \frac{ \max \limits_{i=0:k}\{\|B_i\|_2\}}{ \min\{\|B_0\|_2, \|B_k\|_2\}}, \ \ \
		\widetilde{\rho}:=\frac{ \max \limits_{i=0:k}\{\|\widetilde B_i\|_2\}}{ \min\{\|\widetilde B_0\|_2, \|\widetilde B_k\|_2\}}.
		\end{equation}
		\begin{enumerate}
			\item If $k=1$, then
			$$\frac{1}{4 \; \mathrm{cond}_{\infty}(A)\, \widetilde{\rho}} \leq Q_{\theta}^{r}\leq  4\, \mathrm{cond}_{\infty}(A) \, \rho.$$
			\item If $k\geq 2$, then
			$$\frac{1}{Z_k \, \mathrm{cond}_{\infty}(A)^{k-1} \, \widetilde{\rho}}\leq
			Q_{\theta}^{r}\leq  Z_k\; \mathrm{cond}_{\infty}(A)^{k-1}\, \rho.$$
		\end{enumerate}
	\end{theorem}
	
	\begin{proof}We only prove the upper bounds, since the lower bounds can be obtained
		from the upper bounds using a similar argument to that used in \eqref{eq.invlowbounds}.
		
		Let $A =\left[ \begin{array}{cc} a & b \\c & d \end{array} \right]$. Select an arbitrary representative $[\alpha_0 , \beta_0]^T$ of $(\alpha_0 , \beta_0)$, and consider the representative $[\gamma_0 , \delta_0]^T :=A^{-1}[\alpha_0, \beta_0]^T=\left [ \frac{d\alpha_0-b\beta_0}{\mathrm{det}(A)}, \frac{a\beta_0-c\alpha_0}{\mathrm{det}(A)} \right ]$ of the eigenvalue of $M_A(P)$ associated with $(\gamma_0, \delta_0)$.

		If $k=1$, then, from Corollary \ref{quotient-expr-2}(3.), \eqref{Qa-k1-1} and \eqref{Qa-k1-2}, we obtain
		\[
		Q_{\theta}^{r}  \leq \frac{ \max \limits_{i=0:k}\{\|\widetilde{B}_i\|_2\}}{ \min\{\|B_0\|_2, \|B_k\|_2\}} \, Q_{\theta}^a \leq \frac{ \max \limits_{i=0:k}\{\|\widetilde{B}_i\|_2\}}{ \min\{\|B_0\|_2, \|B_k\|_2\}} \, 2\, \|A^{-1}\|_{\infty}.
		\]
		Proposition \ref{cotaMobius} implies $\max \limits_{i=0:k}\{\|\widetilde{B}_i\|_2\} \leq 2\, \|A\|_{\infty}\,\max \limits_{i=0:k}\{\|B_i\|_2\}$, which combined with the previous inequality yields the upper bound for $k=1$.
		
		If $k\geq 2$, then, from Corollary \ref{quotient-expr-2}(3.) and the inequalities \eqref{casek2-2} and \eqref{casek2-4}, we get
		\begin{align} \nonumber
		Q_{\theta}^{r} & \leq \frac{ \max \limits_{i=0:k}\{\|\widetilde{B}_i\|_2\}}{ \min\{\|B_0\|_2, \|B_k\|_2\}} \, \frac{(k+1)\ }{|\mathrm{det}(A)|}\frac{\max\{|\gamma_0|^k, |\delta_0|^k\}}{(|\alpha_0|^{k}+|\beta_0|^{k})}\frac{2 \max\{|\beta_0|^2,|\alpha_0|^2\}}{\max\{|\delta_0|^2,|\gamma_0|^2\}} \\ & \leq \frac{ \max \limits_{i=0:k}\{\|\widetilde{B}_i\|_2\}}{ \min\{\|B_0\|_2, \|B_k\|_2\}} \,
		2(k+1) \, \frac{ \|A^{-1}\|_{\infty}^{k-1}}{\|A\|_{1} }, \nonumber
		\end{align}
		which combined with Proposition \ref{cotaMobius} and $\|A\|_1 \geq \|A\|_\infty/2$ yields the upper bound for $k\geq 2$.
	\end{proof}

	\begin{remark}\label{bounds2-Qhr} (Discussion on the bounds in Theorem \ref{main-homo})
		The only difference between the bounds in Theorem \ref{main-homo} and those in Theorem \ref{main-homo1} is that the former can be obtained from the latter  by   multiplying the upper bounds by $\rho$ and  dividing the lower bounds by $\widetilde \rho$. Moreover, since $\rho \geq 1$ and $\widetilde\rho \geq 1$, the bounds in Theorem \ref{main-homo} are moderate numbers if and only if the ones in Theorem \ref{main-homo1} are and $\rho \approx 1 \approx \widetilde\rho$. Thus, ignoring again the factor $Z_k$, the three conditions $\mathrm{cond}_{\infty} (A) \approx 1$,  $\rho \approx 1$, and $\widetilde\rho \approx 1$ are sufficient to imply that all the bounds in Theorem \ref{main-homo} are moderate numbers and guarantee that the M\"{o}bius
		transformation $M_A$ \emph{does not change significantly the relative eigenvalue condition numbers of any eigenvalue of a matrix polynomial $P$ satisfying   $\rho \approx 1$ and $\widetilde\rho \approx 1$}. Note that the presence of $\rho$ and $\widetilde\rho$ is natural, since $\rho$ has appeared previously in a number of results that compare the relative eigenvalue condition numbers of a matrix polynomial and of some of its linearizations \cite{HigMacTis,bueno-calcolo-2018}.
	\end{remark}
	
	
	\subsection{Bounds involving eigenvalues for M\"obius transformations induced by ill-conditioned matrices} \label{subsec.53} The bounds in Theorem \ref{main-homo0} on $Q_{\theta}^a$ are very satisfactory under the sufficient conditions $\|A\|_{\infty} \approx \|A^{-1}\|_{\infty} \approx 1$ since, then, the lower and upper bounds are moderate numbers not far from one for moderate values of $k$. The same happens with the bounds on $Q_{\theta}^p$ in Theorem \ref{main-homo1} under the sufficient condition $\mathrm{cond}_{\infty} (A) \approx 1$, and for the bounds in Theorem \ref{main-homo} on $Q_{\theta}^r$, with the two additional conditions $\rho \approx \widetilde\rho \approx 1$. Obviously, these bounds are no longer satisfactory if $\mathrm{cond}_{\infty} (A) \gg 1$, i.e., if the M\"obius transformation is induced by an ill-conditioned matrix, since the lower and upper bounds are very different from each other and do not give any information about the true values of $Q_{\theta}^a$, $Q_{\theta}^p$, and $Q_{\theta}^r$. Note, in particular, that, for any ill-conditioned $A$, the upper bounds in Theorems \ref{main-homo1} and \ref{main-homo} are much larger than $1$, while the lower bounds are much smaller than $1$. 
	
	Although we do not know any M\"obius transformation $M_A$ with $\mathrm{cond}_{\infty} (A) \gg 1$ that is useful in applications and we do not see currently any reason for using such transformations, we consider them in this section for completeness and in connection with the attainability of the bounds in such situation.
	
	For brevity, we limit our discussion to the bounds on the quotients $Q_{\theta}^a$ and $Q_{\theta}^p$, since the presence of $\rho$ and $\widetilde\rho$ in Theorem \ref{main-homo} complicates the discussion on $Q_{\theta}^r$. 
	
	We start by obtaining in Theorem \ref{th.sharperquotient} 
	sharper upper and lower bounds on $Q_{\theta}^a$ and $Q_{\theta}^p$ at the cost of involving the eigenvalues in the expressions of the new  bounds. The reader will notice that in Theorem \ref{th.sharperquotient}, we are using the 1-norm for degree $k=1$ and the $\infty$-norm for degree $k\geq 2$. The reason for these different choices of norms is that they lead to sharper bounds in each case. Obviously, in the case $k=1$, we can also use the  $\infty$-norm at the cost of worsening somewhat the bounds on $Q_{\theta}^a$ and $Q_{\theta}^p$.
	
	\begin{theorem}\label{th.sharperquotient}
		Let $P(\alpha,\beta)= \sum_{i=0}^k \alpha^i \beta^{k-i} B_i \, {\in \mathbb{C}[\alpha,\beta]^{n\times n}_k}$ be a regular homogeneous matrix polynomial and let $A=\left[\begin{array}{cc} a & b \\ c & d\end{array} \right] \in GL(2, \mathbb{C})$. Let $M_A(P)(\gamma, \delta)= \sum_{i=0}^k \gamma^i \delta^{k-i} {\widetilde{B}_i  \in \mathbb{C}[\alpha,\beta]^{n\times n}_k}$ be the M\"obius transform of $P(\alpha, \beta)$ under $M_A$. Let $(\alpha_0, \beta_0)$ be a simple eigenvalue of $P(\alpha, \beta)$ and let $\langle A^{-1} [\alpha_0, \beta_0]^T \rangle $ be the eigenvalue of $M_A(P)$ associated with $(\alpha_0, \beta_0)$. Let $[\alpha_0, \beta_0]^T$ be an arbitrary representative of   $(\alpha_0, \beta_0)$ and let $[\gamma_0, \delta_0]^T  :=A^{-1}[\alpha_0, \beta_0]^T$  be the associated representative of $\langle A^{-1} [\alpha_0, \beta_0]^T \rangle$. Let $Q_{\theta}^a$ and $Q_{\theta}^p$ be the quotients in Definition \ref{notation-hom3}(1.) and (2.), respectively. 
		\begin{enumerate}
			\item If $k=1$, then 
			\begin{eqnarray*}
				\frac{1}{2|\mathrm{det}(A)|}  \frac{\|[ \alpha_0, \beta_0]^T\|_{1}}{\|[\gamma_0, \delta_0]^T\|_{1}}  \leq &  Q_{\theta}^{a} &\leq   \frac{2}{|\mathrm{det}(A)|} \frac{\|[ \alpha_0, \beta_0]^T\|_{1}}{\|[\gamma_0, \delta_0]^T\|_{1}}, \\
				\frac{1}{2|\mathrm{det}(A)|}  \frac{\|[ \alpha_0, \beta_0]^T\|_{1}}{\|[\gamma_0, \delta_0]^T\|_{1}}  \, \frac{\max \limits_{i=0:k}\{\|\widetilde B_i\|_2\}}{\max \limits_{i=0:k}\{\|B_i\|_2\}} \leq &  Q_{\theta}^{p} &\leq   \frac{2}{ |\mathrm{det}(A)|} \frac{\|[ \alpha_0, \beta_0]^T\|_{1}}{\|[\gamma_0, \delta_0]^T\|_{1}} \, \frac{\max \limits_{i=0:k}\{\|\widetilde B_i\|_2\}}{\max \limits_{i=0:k}\{\|B_i\|_2\}}.
			\end{eqnarray*}
			\item If $k\geq 2$, then 
		\end{enumerate}
		\begin{align*}
		& \frac{1}{2(k+1)\, |\mathrm{det}(A)|}  \left(\frac{\|[ \gamma_0, \delta_0]^T\|_{\infty}}{\|[\alpha_0, \beta_0]^T\|_{\infty}}\right)^{k-2}  \leq   Q_{\theta}^{a} \leq   \frac{2(k+1)\ }{|\det(A)|} \left(\frac{\|[ \gamma_0, \delta_0]^T\|_{\infty}}{\|[\alpha_0, \beta_0]^T\|_{\infty}}\right)^{k-2}, \\
		& \phantom{aaaaaaaaa} \frac{1}{2(k+1)\, |\det(A)|}  \left(\frac{\|[ \gamma_0, \delta_0]^T\|_{\infty}}{\|[\alpha_0, \beta_0]^T\|_{\infty}}\right)^{k-2}    \, \frac{\max \limits_{i=0:k}\{\|\widetilde B_i\|_2\}}{\max \limits_{i=0:k}\{\|B_i\|_2\}} \leq Q_{\theta}^{p},\\
		&  \phantom{aaaaaaaaa} Q_{\theta}^{p} \leq  \frac{2(k+1)\ }{|\det(A)|} \left(\frac{\|[ \gamma_0, \delta_0]^T\|_{\infty}}{\|[\alpha_0, \beta_0]^T\|_{\infty}}\right)^{k-2}  \, \frac{\max \limits_{i=0:k}\{\|\widetilde B_i\|_2\}}{\max \limits_{i=0:k}\{\|B_i\|_2\}}. 
		\end{align*}
		Moreover, the bounds in this theorem are sharper than those in Theorems \ref{main-homo0} and \ref{main-homo1}. That is, each upper (resp. lower) bound in the previous inequalities is smaller (resp. larger) than or equal to the corresponding upper (resp. lower) bound in Theorems \ref{main-homo0} and \ref{main-homo1}.
	\end{theorem}
	
	\begin{proof}
		We only need to prove the bounds for $Q^a_{\theta}$. The bounds for $Q^p_{\theta}$ follow immediately from the bounds for $Q^a_{\theta}$ and the  equality in \eqref{eq.factatop}. For $k=1$, the upper bound on $Q^a_{\theta}$ can be obtained from \eqref{Qa-k1-2}; 
		the lower bound follows easily from  \eqref{Qa-k1-1} through an argument similar to the one leading to the upper bound. For $k\geq 2$, the upper bound on $Q^a_{\theta}$ is just \eqref{casek2-3}, and the lower bound follows easily from Corollary \ref{quotient-expr-2}(1.) and \eqref{Qrah} through an argument similar to the one leading to the upper bound.
		
		Next, we prove that the bounds in this theorem are sharper than those in Theorems \ref{main-homo0} and \ref{main-homo1}. For the upper bounds in Theorem \ref{main-homo0}, this follows from \eqref{Qa-k1-2} for $k=1$ and the inequality  \eqref{casek2-4} for $k\geq 2$. The corresponding results for the lower bounds in Theorem \ref{main-homo0} follow from a similar argument. Note that the bounds in Theorem \ref{main-homo1} can be obtained from the ones in this theorem in two steps: first bounds on $\|[ \alpha_0, \beta_0]^T\|_{1} / \|[\gamma_0, \delta_0]^T\|_{1}$, for $k=1$, and on $\|[ \gamma_0, \delta_0]^T\|_{\infty}/\|[\alpha_0, \beta_0]^T\|_{\infty}$, for $k\geq 2$, are obtained  and, then, upper and lower bounds on  \\$\max \limits_{i=0:k}\{\|\widetilde B_i\|_2\} / \max \limits_{i=0:k}\{\|B_i\|_2\}$ are obtained from Proposition \ref{cotaMobius} (the lower bounds are obtained by interchanging the roles of $B_i$ and $\widetilde{B}_i$ and by replacing $A$ by $A^{-1}$, since $P = M_{A^{-1}} (M_A (P))$). This proves that the bounds in this theorem are sharper than those in Theorems \ref{main-homo0} and \ref{main-homo1}.
	\end{proof}
	
	Observe that, from Theorem \ref{th.sharperquotient}, we obtain that, disregarding the (pessimistic and moderate) factors in the bounds depending only on the degree $k$,
	\begin{equation} \label{eq.1approxseillcond}
	Q_{\theta}^{a} \approx \frac{1}{|\mathrm{det}(A)|}\frac{\|[ \alpha_0, \beta_0]^T\|_{1}}{\|[\gamma_0, \delta_0]^T\|_{1}}, \:  Q_{\theta}^{p} \approx  \frac{1}{|\mathrm{det}(A)|}\frac{\|[ \alpha_0, \beta_0]^T\|_{1}}{\|[\gamma_0, \delta_0]^T\|_{1}} \, \frac{\max \limits_{i=0:k}\{\|\widetilde B_i\|_2\}}{\max \limits_{i=0:k}\{\|B_i\|_2\}}, \quad \mbox{for $k=1$},
	\end{equation}
	and
	\begin{align} \label{eq.2approxseillcond}
	Q_{\theta}^{a} &\approx   \frac{1}{|\det(A)|} \left(\frac{\|[ \gamma_0, \delta_0]^T\|_{\infty}}{\|[\alpha_0, \beta_0]^T\|_{\infty}}\right)^{k-2},  \qquad \mbox{for $k\geq 2$}, \\ \label{eq.3approxseillcond} Q_{\theta}^{p} & \approx  \frac{1}{|\det(A)|} \left(\frac{\|[ \gamma_0, \delta_0]^T\|_{\infty}}{\|[\alpha_0, \beta_0]^T\|_{\infty}}\right)^{k-2}  \, \frac{\max \limits_{i=0:k}\{\|\widetilde B_i\|_2\}}{\max \limits_{i=0:k}\{\|B_i\|_2\}},  \qquad \mbox{for $k\geq 2$}.
	\end{align}
	
	The approximate equalities \eqref{eq.1approxseillcond}, \eqref{eq.2approxseillcond}, and \eqref{eq.3approxseillcond} are much simpler than the exact expressions for the quotients $Q^a_{\theta}$ and $Q^p_{\theta}$ given by (\ref{Qrah}), using the appropriate weights (see Corollary \ref{quotient-expr-2}), and reveal clearly when the lower and upper bounds in Theorems \ref{main-homo0} and \ref{main-homo1} are attained. An analysis of these approximate expressions leads to some interesting conclusions that are informally discussed below. Throughout this  discussion, we often use expressions similar to ``this bound is essentially attained'' with the meaning that it is attained disregarding the factors $S_k$ and $Z_k$ appearing in Theorems \ref{main-homo0} and \ref{main-homo1}. Also, for brevity, when comparing the bounds for $k=1$ and $k\geq 2$ in our analysis,  we  use the fact 
	$$\frac{\|[\alpha_0, \beta_0]^T\|_1}{\|[\gamma_0, \delta_0]^T\|_1} \approx \frac{\|[\alpha_0, \beta_0]^T\|_{\infty}}{\|[\gamma_0, \delta_0]^T\|_{\infty}}$$
	without saying it explicitly.

	\begin{enumerate}[leftmargin=*]
		\item \emph{The bounds in Theorem \ref{main-homo0} on $Q_{\theta}^a$ are essentially optimal} in the following sense: for a fixed matrix $A$ (which is otherwise arbitrary, and so, it may be very ill-conditioned), it is always possible to find  regular matrix polynomials with simple eigenvalues for which the upper bounds are essentially attained; the same happens with the lower bounds. Next we show these facts. 

		For $k=1$, \eqref{eq.1approxseillcond} implies that the upper (resp. lower) bound in Theorem \ref{main-homo0} is essentially attained for any regular pencil with a simple eigenvalue $(\alpha_0, \beta_0)$ satisfying 
		\begin{equation} \label{eq.attainb1}
		\frac{\|[ \alpha_0, \beta_0]^T\|_{1}}{\|[\gamma_0, \delta_0]^T\|_{1}}  = \frac{\|A A^{-1}[ \alpha_0, \beta_0]^T\|_{1}}{\| A^{-1}[\alpha_0, \beta_0]^T\|_{1}} = \|A\|_{1}.
		\end{equation}
		In contrast, the lower bound is essentially attained by any regular pencil with a simple eigenvalue $(\alpha_0, \beta_0)$ such that \begin{equation}\label{condexact2}\frac{\| A^{-1} [ \alpha_0, \beta_0]^T\|_{1}}{ \|[\alpha_0, \beta_0]^T\|_{1}} = \|A^{-1}\|_1.
		\end{equation}
		Note that,  for any positive integer $n$, a regular pencil of size $n\times n$ can be easily constructed satisfying \eqref{eq.attainb1} (resp. \eqref{condexact2}): just take a diagonal pencil with a main-diagonal entry having the desired eigenvalue as a root.
		
		For $k=2$, \eqref{eq.2approxseillcond} implies that $Q_{\theta}^a \approx 1/|\det(A)| = \|A^{-1}\|_{\infty}/\|A\|_1$. So, the quotient $Q_{\theta}^a$ is  independent of the eigenvalue and the polynomial's matrix coefficients, and is essentially always equal to both the lower and upper bounds.
		
		Finally, for $k>2$, \eqref{eq.2approxseillcond} implies that the upper (resp. lower) bound in Theorem \ref{main-homo0} is essentially attained if  the  right (resp. left) inequality in
		$$\frac{1}{\|A\|_{\infty}} \leq \frac{\|[\gamma_0, \delta_0]^T\|_{\infty}}{\|[\alpha_0, \beta_0]^T\|_{\infty}} \leq \|A^{-1}\|_{\infty}$$
		is an equality. 
		Again,  for any size $n \times n$,  regular matrix polynomials with simple eigenvalues satisfying either of the two conditions can be easily constructed as diagonal matrix polynomials of degree $k$ having a main diagonal entry with the desired eigenvalue as a root.
		
		\item  From \eqref{eq.1approxseillcond} and \eqref{eq.2approxseillcond}, and the discussion above, we see that, for a fixed ill-conditioned matrix $A$ (which implies that  the upper and  lower bounds on $Q_{\theta}^a$ in Theorem \ref{main-homo0} are very far apart), \emph{the behaviours of $Q^a_{\theta}$ for $k=1$, $k=2$, and $k>2$ are very different from each other} in the following sense:  
		If the lower (resp. upper) bound on $Q^a_{\theta}$ given in Theorem \ref{main-homo0}  is essentially attained for an eigenvalue $(\alpha_0, \beta_0)$ when $k=1$, then the upper (resp. lower) bound on $Q^a_{\theta}$ is attained for the same eigenvalue when $k>2$ (recall that the expression for $Q_{\theta}^a$ only depends on $A$, the eigenvalue $(\alpha_0, \beta_0)$ and the degree $k$ of the matrix polynomial; also recall that for any $k$ we can construct a matrix polynomial of degree $k$ with a simple eigenvalue equal to $(\alpha_0, \beta_0)$). When $k=2$, the true value of $Q^a_{\theta}$ does not depend (essentially) on  $(\alpha_0, \beta_0)$, according to \eqref{eq.2approxseillcond}.  In this sense, the behaviours for $k=1$ and $k >2$ are opposite from each other, while the one for $k=2$ can be seen as ``neutral''.
		
		\item    \emph{The bounds in Theorem \ref{main-homo1} on $Q_{\theta}^p$ are essentially optimal} in the following sense: if the matrix $A$ is fixed, then it is always possible to find  regular matrix polynomials with a simple eigenvalue for which the upper bounds on $Q_{\theta}^p$ are essentially attained; the same happens with the lower bounds. 
		
		Here we only discuss our claim for the upper bounds on $Q_{\theta}^p$. Then, to show that the lower bounds on $Q_{\theta}^p$ can be attained,  an argument similar to that in \eqref{eq.invlowbounds} can be used.
		
		From \eqref{eq.1approxseillcond} and \eqref{eq.3approxseillcond}, for the upper bound on $Q_{\theta}^p$ to be attained, both the upper bound on $Q_{\theta}^a$ (in Theorem \ref{main-homo0}) and the upper bound on \\$\max \limits_{i=0:k}\{\|\widetilde B_i\|_2\}/\max \limits_{i=0:k}\{\|B_i\|_2\}$ (in Proposition \ref{cotaMobius}) must be attained. Thus, we need to construct a regular matrix polynomial with a simple eigenvalue for which both bounds are attained simultaneously. 
		In our discussion in item 1. above we discussed how to find an eigenvalue $(\alpha_0, \beta_0)$ attaining the upper bound on $Q_{\theta}^a$ for each value of $k$.
		In order to construct a regular matrix polynomial $P$ with $(\alpha_0, \beta_0)$ as a simple eigenvalue and such that $\max \limits_{i=0:k}\{\|\widetilde B_i\|_2\}/\max \limits_{i=0:k}\{\|B_i\|_2\} \approx \|A\|_{\infty}^k$ we proceed as follows:
		
		Let $q(\alpha,\beta)$ be any nonzero scalar polynomial of degree $k$ such that $q(\alpha_0,\beta_0) = 0$ and define $P(\alpha, \beta) = \mathrm{diag}(\varepsilon q(\alpha,\beta), Q(\alpha, \beta))=: \sum_{i=0}^k \alpha^i \beta^{k-i} B_i$,  where $\varepsilon>0$ is an arbitrarily small parameter and $Q(\alpha, \beta) = \sum_{i=0}^k \alpha^i \beta^{k-i} C_i \in \allowbreak \mathbb{C}[\alpha,\beta]^{(n-1) \times (n-1)}_k$ is a regular matrix polynomial. Then, $P(\alpha , \beta)$ is regular, and has $(\alpha_0, \beta_0)$ as a simple eigenvalue if $(\alpha_0, \beta_0)$ is not an eigenvalue of $Q(\alpha, \beta)$. Moreover, if $\varepsilon$ is sufficiently small and $\|C_\ell \|_2 := \max \limits_{i=0:k}\{\|C_i\|_2\}$, then $\max \limits_{i=0:k}\{\|B_i\|_2\} = \|B_\ell\|_2$. Let us assume, for simplicity, that  $(\alpha_0, \beta_0) \ne (1,0)$ and $(\alpha_0, \beta_0) \ne (0,1)$, although such assumption is not essential. Next, we explain how to construct $Q(\alpha, \beta)$ depending on which entry of $A$  has the largest modulus. 
		
		$\bullet$ If $\|A\|_M = |a|$, let  $Q(\alpha, \beta) := \alpha^k C_k$, where  $C_k$   is an arbitrary $(n-1) \times (n-1)$ nonsingular matrix such that,  for $\varepsilon$ small enough, $P(\alpha , \beta)$ satisfies $\|B_k\|_2 \gg \|B_i\|_2$ for $i\ne k$, and  so, \eqref{Bl} implies $\widetilde{B}_k  \approx a^k B_k$. Hence $\max_{i=0:k}\{\|\widetilde{B}_i\|_2\} \geq  \|\widetilde{B}_k\|_2 \approx |a|^k \|B_k\|_2$. By \eqref{normBi}, we have 
		$$\frac{1}{2^k} \, \|A\|_{\infty}^k \leq \|A\|_M^k \leq \frac{\max \limits_{i=0:k}\{\|\widetilde B_i\|_2\}}{\max \limits_{i=0:k}\{\|B_i\|_2\}} \leq \|A\|_{\infty}^k (k+1) {k \choose \lfloor k/2 \rfloor}. $$
		Thus, we deduce that $\max \limits_{i=0:k}\{\|\widetilde B_i\|_2\}/\max \limits_{i=0:k}\{\|B_i\|_2\} \approx \|A\|_{\infty}^k$, up to a factor depending on $k$.
		Note that $(\alpha_0, \beta_0)$ is not an eigenvalue of $Q(\alpha, \beta)$ by construction.
		
		$\bullet$ If $\|A\|_M = |b|$, then the same conclusion follows by taking again $Q(\alpha, \beta) = \alpha^k C_k$, since \eqref{Bl} implies $\widetilde{B}_0 \approx b^k B_k$. 
		
		$\bullet$ If $\|A\|_M  = |c|$, we get the desired result from taking $Q(\alpha, \beta) = \beta^k C_0$ , since \eqref{Bl} implies $\widetilde{B}_k \approx c^k B_0$. 
		
		$\bullet$ If $\|A\|_M = |d|$, take again $Q(\alpha, \beta) = \beta^k C_0$ , since \eqref{Bl} implies $\widetilde{B}_0 \approx d^k B_0$.

		\item  From \eqref{eq.1approxseillcond} and \eqref{eq.3approxseillcond}, we see that, for a fixed ill-conditioned $A$, \emph{the behaviours of $Q^p_{\theta}$ for $k=1$ and $k > 2$ are very different from each other} in the following sense:  the eigenvalues $(\alpha_0 , \beta_0)$ for which $Q_{\theta}^p$ essentially attains the upper (resp. lower) bound  given in Theorem \ref{main-homo1} for $k>2$, do not attain the upper (resp. lower) bound on $Q^p_{\theta}$ for $k=1$. Notice, for example, that if $Q_{\theta}^p$ attains the upper bound for some polynomial of degree $k>2$ having $(\alpha_0, \beta_0)$ as a simple eigenvalue, then $\frac{\|[\gamma_0, \delta_0]^T\|_{\infty}}{\|[\alpha_0, \beta_0]^T\|_{\infty}}\approx \|A^{-1}\|_{\infty}$ and $\max \limits_{i=0:k}\{\|\widetilde B_i\|_2\}/\max \limits_{i=0:k}\{\|B_i\|_2\}\approx \|A\|_{\infty}^{k}$, which implies that
		$Q_{\theta}^p \approx \mbox{cond}_{\infty}(A)^{k-1}$ by \eqref{quotient-norm-eig} while, in this case, the value of $Q_{\theta}^p$ associated with a polynomial of degree 1 is  of order 1 (by \eqref{eq.1approxseillcond}), which is not close to the upper bound $4\, \mbox{cond}_{\infty}(A)$ since $A$ is ill-conditioned. 
		Note also that, in contrast to the discussion for $Q^a_{\theta}$, we cannot state that such behaviours are opposite from each other. In our example, the lower bound for $Q_{\theta}^p$ with $k=1$ is much smaller than 1 when $\mbox{cond}_{\infty}(A)$ is very large while $Q_{\theta}^p$ may be of order 1. 
		These different behaviours have been very clearly observed in the numerical experiments presented in Section \ref{sec:num} as it is explained in the next paragraph.
		
		\item  For a fixed ill-conditioned $A$, we have observed numerically that the eigenvalues $(\alpha_0, \beta_0)$ of randomly generated matrix polynomials $P(\alpha , \beta)$ of any degree satisfy, almost always, that 
		\begin{equation} \label{eq.attainb2}
		\frac{\|[\gamma_0, \delta_0]^T\|_{\infty}}{\|[ \alpha_0, \beta_0]^T\|_{\infty}} = \theta \| A^{-1} \|_{\infty}, 
		\end{equation}
		with $\theta$ not too close to $0$. This is naturally expected because ``random'' vectors $[ \alpha_0, \beta_0]^T$, when expressed in the (orthonormal) basis of right singular vectors of $A^{-1}$, have non-negligible components on the vector corresponding to the largest singular value. We have also observed that  randomly generated polynomials $P(\alpha , \beta)$ of moderate degree satisfy, almost always,
		\begin{equation} \label{eq.attainb3}
		\frac{\max \limits_{i=0:k}\{\|\widetilde B_i\|_2\}}{\max \limits_{i=0:k}\{\|B_i\|_2\}} = \xi \|A \|_{\infty}^k,
		\end{equation}
		with $\xi$ not far from $1$. Combining \eqref{eq.attainb2}, \eqref{eq.attainb3}, \eqref{eq.1approxseillcond}, and \eqref{eq.3approxseillcond} we get that, for randomly generated polynomials, the following conditions almost always hold: $Q_{\theta}^p \approx \xi/\theta \approx 1$ for $k=1$; $Q_{\theta}^p \approx \xi \|A \|_{\infty}^2 /|\det(A)| \approx \mathrm{cond}_{\infty} (A)$ for $k=2$; and $Q_{\theta}^p \approx \xi  \theta^{k-2} \|A^{-1}\|_{\infty}^{k-2} \|A \|_{\infty}^{k}  /|\det(A)| \approx \mathrm{cond}_{\infty} (A)^{k-1}$ for $k>2$. This explains why in random numerical tests for $k=1$ the quotient $Q_{\theta}^p$ is almost 
		always close to $1$ and seems to be insensitive to the conditioning of $A$, as we will check numerically in Section \ref{sec:num}. However, remember, that both the upper and lower bounds in Theorem \ref{main-homo1} can be essentially attained for any fixed $A$.
	\end{enumerate}

	We finish this section by remarking that the differences mentioned above between the degrees $k=1$ and $k\geq 2$ are also observed numerically for the relative quotients $Q_{\theta}^r$ as shown in Section \ref{sec:num}, although the differences are somewhat less clear. It is also possible to argue that the lower and upper bounds in Theorem \ref{main-homo} can be essentially attained, but the arguments need to take into account the factors $\rho$ and $\widetilde\rho$ and are more complicated. We have performed numerical tests that confirm that those bounds are approximately attainable.
}

\section{Effect of M\"{o}bius transformations on backward errors of approximate eigenpairs} \label{sec.backwerrors} 

The scenario in this section is the following: we want to compute eigenpairs of a regular homogeneous matrix polynomial $P(\alpha,\beta)\in \mathbb{C}[\alpha , \beta]_k^{n\times n}$, but, for some reason, it is advantageous to compute eigenpairs of its M\"{o}bius transform $M_A (P) (\gamma , \delta)$, where $A =\left[ \begin{array}{cc} a & b \\c & d \end{array} \right]\in GL(2, \mathbb{C})$. A motivation for this might be, for instance, that $P(\alpha,\beta)$ has a certain structure that can be used for computing very efficiently and/or accurately its eigenpairs, but there are no specific algorithms available for such structure, although there are for the structured polynomial $M_A (P) (\gamma , \delta)$. Note that if $(\widehat{x}, (\widehat{\gamma_0}, \widehat{\delta_0}))$ and $(\widehat{y}^*,(\widehat{\gamma_0},\widehat{\delta_0}))$ are computed   \emph{approximate} right and left eigenpairs of $M_A(P)$, and $(\widehat{\alpha_0},\widehat{\beta_0}) : = (a\widehat{\gamma_0} + b \widehat{\delta_0},c\widehat{\gamma_0} + d \widehat{\delta_0})$ then, because of Proposition \ref{prop-mobius} 
and Lemma \ref{eig-hom}, $(\widehat{x},(\widehat{\alpha_0},\widehat{\beta_0}))$ and $(\widehat{y^*},(\widehat{\alpha_0},\widehat{\beta_0}))$ can be considered  approximate right and left eigenpairs of $P(\alpha, \beta)$. Assuming that $(\widehat{x}, (\widehat{\gamma_0}, \widehat{\delta_0}))$ and $(\widehat{y}^*,(\widehat{\gamma_0},\widehat{\delta_0}))$ have been computed with small backward errors in the sense of Definition \ref{def-backw}, a natural question in this setting is whether $(\widehat{x},(\widehat{\alpha_0},\widehat{\beta_0}))$ and $(\widehat{y}^*,(\widehat{\alpha_0},\widehat{\beta_0}))$ are also approximate eigenpairs of $P$ with small backward errors. This would happen if the quotients
\begin{equation}
\label{quot-backw} Q_{\eta , \mathrm{right}} :=
\frac{\eta_{P} (\widehat{x}, (\widehat{\alpha_0}, \widehat{\beta_0}) )}{\eta_{M_A(P)}(\widehat{x}, (\widehat{\gamma_0}, \widehat{\delta_0}))},\quad  Q_{\eta , \mathrm{left}} :=
\frac{\eta_{P} (\widehat{y}^*,(\widehat{\alpha_0}, \widehat{\beta_0}) )}{\eta_{M_A(P)}(\widehat{y}^*, (\widehat{\gamma_0}, \widehat{\delta_0}))} \, 
\end{equation}
are moderate numbers not much larger than one. In this section we provide upper bounds on the quotients in \eqref{quot-backw} that allow us to determine simple sufficient conditions that guarantee that such quotients are not large numbers. For completeness, we also provide lower bounds for these quotients, although they are less interesting than the upper ones in the scenario described above.

Note that, from Theorem \ref{back-def}, we can easily deduce that the backward error is independent of the choice of  representative of the approximate eigenvalue.

The first result in this section is Theorem \ref{eta-exact}, which proves that the  quotients in \eqref{quot-backw} are equal and provides an explicit expression for them.

\begin{theorem}
	\label{eta-exact}
	Let $P(\alpha,\beta)=\sum_{i=0}^k \alpha^i \beta^{k-i}B_i \in \mathbb{C} [\alpha , \beta]_k^{n\times n}$ be a regular homogeneous matrix polynomial, let $A =\left[ \begin{array}{cc} a & b \\c & d \end{array} \right]\in GL(2, \mathbb{C})$, and let $M_A(P) (\gamma, \delta) = \sum_{i=0}^k \gamma^i \delta^{k-i} \widetilde{B}_i$ be the M\"obius transform of $P(\alpha , \beta)$ under $M_A$.  Let $(\widehat{x}, (\widehat{\gamma_0}, \widehat{\delta_0}))$ and $(\widehat{y}^*, (\widehat{\gamma_0}, \widehat{\delta_0}))$ be approximate right and left eigenpairs of $M_A(P)$, and let $[\widehat{\alpha_0},\widehat{\beta_0}]^T : = A [ \widehat{\gamma_0},  \widehat{\delta_0}]^T$. Let $Q_{\eta , \mathrm{right}}$ and $Q_{\eta , \mathrm{left}}$ be as in \eqref{quot-backw} and let $\omega_i$ and $\widetilde{\omega_i}$ be the weights used in the definition of the backward errors for $P$ and $M_A(P)$, respectively. Then, 
	\begin{equation} \label{eq.exactbackquot}
	Q_{\eta , \mathrm{right}} = Q_{\eta , \mathrm{left}} =\frac{\sum_{i=0}^k|\widehat{\gamma_0}|^i|\widehat{\delta_0}|^{k-i}\widetilde{\omega_i}}{\sum_{i=0}^k |\widehat{\alpha_0}|^i | \widehat{\beta_0}|^{k-i}\omega_i} .
	\end{equation}
	Moreover, \eqref{eq.exactbackquot} is independent of the choice of representative for $(\widehat{\gamma_0}, \widehat{\delta_0})$.
\end{theorem}

\begin{proof}
	Since the backward error does not depend on the choice of representative of approximate eigenvalues, we choose an arbitrary representative $[\widehat{\gamma_0},\widehat{\delta_0}]^T$ of $(\widehat{\gamma_0},\widehat{\delta_0})$, and, once $[\widehat{\gamma_0},\widehat{\delta_0}]^T$ is fixed, we choose  $[\widehat{\alpha_0},\widehat{\beta_0} ]^T: = A[\widehat{\gamma_0}, \widehat{\delta_0}]^T$ as representative of the approximate eigenvalue of $P$. For these representatives note that
	$M_A(P)(\widehat{\gamma_0}, \widehat{\delta_0}) = \sum_{i=0}^k (a \widehat{\gamma_0}+ b \widehat{\delta_0})^i (c \widehat{\gamma_0}+d \widehat{\delta_0})^{k-i} B_i = P(\widehat{\alpha_0}, \widehat{\beta_0})$. Thus, Theorem \ref{back-def} implies \eqref{eq.exactbackquot}.
\end{proof}

Analogously to the quotients of condition numbers in Definition \ref{notation-hom3}, we can consider absolute, relative with respect to the norm of the polynomial, and relative quotients of backward errors. They are defined, taking into account Definition \ref{def.weigths-back}, as 
\begin{equation}
\label{quot-backw-weights}
Q_{\eta, \mathrm{right}}^{s}:=\frac{\eta_{P}^{s} (\widehat{x}, \langle A[\widehat{\gamma_0}, \widehat{\delta_0}]^T\rangle )}{\eta_{M_A(P)}^{s}(\widehat{x}, (\widehat{\gamma_0}, \widehat{\delta_0}))}, \quad 
Q_{\eta , \mathrm{left}}^{s}:=\frac{\eta_{P}^{s} (\widehat{y}^*,  \langle A[\widehat{\gamma_0}, \widehat{\delta_0}]^T\rangle )}{\eta_{M_A(P)}^{s}(\widehat{y}^*, (\widehat{\gamma_0}, \widehat{\delta_0}))},\; \mbox{for $s = a, p, r$}.
\end{equation}

Theorem \ref{eta-bounds} provides upper and lower bounds on the quotients in \eqref{quot-backw-weights}.

\begin{theorem}
	\label{eta-bounds}
	With the same notation and hypotheses of Theorem \ref{eta-exact}, let  $Y_k :=  (k+1)^2 {k \choose \lfloor k/2 \rfloor}$. Then
	\begin{enumerate}
		\item $\displaystyle \frac{1}{(k+1) \, \|A\|_{\infty}^k}  \leq Q_{\eta,\mathrm{right}}^{a} = Q_{\eta,\mathrm{left}}^{a} \leq (k+1) \, \|A^{-1}\|_{\infty}^k$.
		\medskip
		\item $\displaystyle \frac{1}{Y_k \, \mathrm{cond}_{\infty}(A)^k} \leq Q_{\eta,\mathrm{right}}^{p} = Q_{\eta,\mathrm{left}}^{p}  \leq Y_k \, \mathrm{cond}_{\infty}(A)^k$.
		\medskip
		\item If $B_0 \ne 0, B_k \ne 0, \widetilde{B}_0 \ne 0$, and $\widetilde{B}_k \ne 0$, and $\rho$ and $\widetilde{\rho}$ are defined as in \eqref{rho}, then  $$\displaystyle \frac{1}{Y_k \, \mathrm{cond}_{\infty}(A)^k \, \widetilde{\rho}} \leq Q_{\eta,\mathrm{right}}^{r} = Q_{\eta,\mathrm{left}}^{r} \leq Y_k \, \mathrm{cond}_{\infty}(A)^k \, \rho .$$
	\end{enumerate}
\end{theorem}

\begin{proof} We only prove the upper bounds since the lower bounds can be obtained in a similar way. Moreover, we only need to pay attention to the quotients for right eigenpairs, taking into account \eqref{eq.exactbackquot}. Let us start with the absolute quotients. From \eqref{eq.exactbackquot} with $\omega_i = \widetilde{\omega}_i = 1$, we obtain
	\begin{align} \nonumber
	Q_{\eta , \mathrm{right}}^a  & \leq  (k+1) \frac{\|[\widehat{\gamma_0},\widehat{\delta_0}]^T\|_\infty^k}{\|A  [\widehat{\gamma_0},\widehat{\delta_0}]^T \|_\infty^k} \\ &=
	(k+1) \frac{\| A^{-1} A [\widehat{\gamma_0},\widehat{\delta_0}]^T\|_\infty^k}{\|A  [\widehat{\gamma_0},\widehat{\delta_0}]^T\|_\infty^k} 
	\leq (k+1) \|A^{-1} \|_\infty^k. \label{eq.auxxback}
	\end{align}
	
	The upper bound on $Q_{\eta,\mathrm{right}}^{p}$ follows from combining
	$Q_{\eta,\mathrm{right}}^{p}=Q_{\eta,\mathrm{right}}^{a} \frac{\max \limits_{i=0:k}\{\|\widetilde B_i\|_2\} } {\max \limits_{i=0:k}\{\|B_i\|_2\} }$, which is obtained from \eqref{eq.exactbackquot}, the upper bound on $Q_{\eta,\mathrm{right}}^{a}$ obtained above, and \eqref{normBi}.
	
	The upper bound on $Q_{\eta,\mathrm{right}}^{r}$ can be obtained noting that \eqref{eq.exactbackquot} and \eqref{normBi} imply
	\begin{align*}
	Q_{\eta, \mathrm{right}}^{r}
	&\leq  \frac{\max \limits_{i=0:k}\{\|\widetilde{B}_i\|_2\}}{\min  \{\|B_0 \|_2, \|B_k\|_2\}} \frac{\sum_{i=0}^k |\widehat{\gamma_0}|^i|\widehat{\delta_0}|^{k-i}}{ | a\widehat{\gamma_0} + b \widehat{\delta_0}|^k+ |c\widehat{\gamma_0} + d \widehat{\delta_0}|^{k}}\\
	&\leq  (k+1)^2 {k \choose \lfloor k/2\rfloor} \|A\|_{\infty}^k \,  \frac{\|[\widehat{\gamma_0},\widehat{\delta_0}]^T\|_\infty^k}{\|[ a\widehat{\gamma_0} + b \widehat{\delta_0}, c\widehat{\gamma_0} + d \widehat{\delta_0}]^T \|_\infty^k} \; \rho\\
	& \leq (k+1)^2 {k \choose \lfloor k/2\rfloor} \|A\|_{\infty}^k  \|A^{-1}\|_{\infty}^k \, \rho,
	\end{align*}
	where the last inequality is obtained as in \eqref{eq.auxxback}.
\end{proof}

\begin{remark} 
	\label{remark-back}  The bounds in Theorem \ref{eta-bounds} on the quotients of backward errors have the same flavor as those in Theorems \ref{main-homo0}, \ref{main-homo1}, and \ref{main-homo} on the quotients of condition numbers. However, note that in Theorem \ref{eta-bounds} there is no need to make a distinction between the bounds for $k=1$ and $k\geq 2$, in contrast with Theorems \ref{main-homo0}, \ref{main-homo1}, and \ref{main-homo}, since the bounds for the quotients of backward errors are obtained in the same way for all $k$. This has numerical consequences since the differences discussed in Section \ref{subsec.53}, and shown in practice in some of the tests in Section \ref{sec:num}, between the quotients of condition numbers for $k=1$ and $k\geq 2$ when $\mathrm{cond}_{\infty} (A) \gg 1$ do not exist for the quotients of backward errors.  
	
	Ignoring the factors depending only on the degree $k$, Theorem \ref{eta-bounds} guarantees that the quotients of backward errors are moderate numbers under the same sufficient conditions under which Theorems \ref{main-homo0}, \ref{main-homo1}, and \ref{main-homo} guarantee that the quotients of condition numbers are moderate numbers. That is: $\|A\|_{\infty} \approx \|A^{-1}\|_{\infty} \approx 1$ implies  that $Q_{\eta, \mathrm{right}}^{a} = Q_{\eta, \mathrm{left}}^{a}$ is a moderate number,  $\mathrm{cond}_{\infty} (A) \approx 1$ implies that $Q_{\eta, \mathrm{right}}^{p} = Q_{\eta, \mathrm{left}}^{p}$ is a moderate number, and $\mathrm{cond}_{\infty} (A) \approx 1$ and $\rho \approx \widetilde{\rho} \approx 1$ imply that $Q_{\eta, \mathrm{right}}^{r} = Q_{\eta, \mathrm{left}}^{r}$ is a moderate number.
\end{remark}


\section{Numerical experiments}\label{sec:num} 
In this section, we present a few numerical experiments that compare the exact values of the quotients $Q_\theta^p$, $Q_\theta^r$, $Q_{\eta,\mathrm{right}}^{p}$, and $Q_{\eta,\mathrm{right}}^{r}$ with the bounds on these quotients  obtained in Sections \ref{SecHom} and \ref{sec.backwerrors}. Observe that, implicitly, these experiments also compare the exact values of $Q_{\eta,\mathrm{left}}^{p}$ and $Q_{\eta,\mathrm{left}}^{r}$ with the bounds on these quotients as a consequence of Theorem \ref{eta-exact}. We do not present experiments on $Q_\theta^a$ and $Q_{\eta,\mathrm{right}}^{a}$ for brevity and also because the weights corresponding to these quotients are not interesting in applications, as it was explained after Definition \ref{weights}. We remark that many other numerical tests have been performed, in addition to the ones presented in this section, and that all of them confirm the theory developed in this paper.

The results in Sections \ref{SecHom} and \ref{sec.backwerrors} prove that eigenvalue condition numbers and backward errors of approximate eigenpairs can change significantly under M\"{o}bius transformations induced by ill-conditioned matrices. Therefore, the use of such M\"{o}bius transformations is not recommended in numerical practice. As a consequence most of our numerical experiments consider M\"{o}bius transformations induced by matrices $A$ such that $\mathrm{cond}_2 (A) =1$, which implies $1 \leq \mathrm{cond}_\infty (A) \leq 2$. The only exception is Experiment 3. 

Next we explain the goals of each of the numerical experiments in this section. Experiment 1 illustrates that the factor $Z_k$ appearing in the bounds on $Q_\theta^p$ and $Q_\theta^r$ in Theorems \ref{main-homo1} and \ref{main-homo} is very pessimistic in practice. This is a very important fact since $Z_k$ is very large for moderate values of $k$ and, if its effect was observed in practice, then even M\"{o}bius transformations induced by well-conditioned matrices would not be recommendable for  matrix polynomials with moderate degree. Experiment 2 illustrates that $Q_{\theta}^{r}$ indeed depends on the factor $\rho$ defined in \eqref{rho} and, so, that the bounds in Theorem \ref{main-homo} reflect often the behaviour of $Q_{\theta}^{r}$ when $\rho$ is large. Experiment 3 is mainly of academic interest, since it considers M\"{o}bius transformations induced by ill-conditioned matrices. The goal of this experiment is to illustrate the results presented in Subsection \ref{subsec.53}, in particular, the different typical behaviors of the quotients $Q_{\theta}^p$ for $k=1$ and $k\geq 2$ when the polynomials are randomly generated. Experiments 4 and 5 are the counterparts of Experiments 1 and 2, respectively, for the quotients of backward errors. 

All the experiments have been run on MATLAB-R2018a. Since in these experiments we have  sometimes encountered badly scaled matrix polynomials (that is, polynomials with matrix coefficients whose norms vary widely), ill-conditioned eigenvalues have appeared. These eigenvalues could potentially be computed inaccurately and spoil the comparison between the results in the experiments and the theory. To avoid this problem,  all the computations in Experiments 1, 2, and 3 have been done using variable precision arithmetic with 40 decimal digits of precision. To obtain the eigenvalues of each matrix polynomial $P$ in these experiments, the function {\texttt eig} in MATLAB has been applied to the first Frobenius companion form of $P$. In Experiments 4 and 5, we have also used variable precision arithmetic with 40 decimal digits of precision for computing the M\"{o}bius transforms of the generated polynomials, but, since we are dealing with backward errors, the eigenvalues have been computed in the standard double precision of MATLAB with the command {\texttt polyeig}. 

\textbf{Experiment 1.} In this experiment, we generate random matrix polynomials $P(\alpha,\beta)=\sum_{i=0}^k\alpha^i\beta^{k-i}B_i$ by using the MATLAB's command {\texttt randn} to generate the matrix coefficents $B_i$. Then, for each polynomial $P(\alpha,\beta)$, a random $2\times 2$ matrix $A$ is constructed as the unitary $Q$ matrix produced by the command {\texttt qr(randn(2))}, which guarantees that $\mathrm{cond}_{2}(A)=\|A\|_2\|A^{-1}\|_2=1$. Finally, the M\"{o}bius transform $M_A (P)$ is computed. We have worked with degrees $k=1:15$ and, for each degree $k$, we have generated $n_k$ matrix polynomials of size $5\times 5$, where the values of $n_k$ can be found in the following table:
\begin{equation}
\label{table}
\begin{array}{|c|c|c|c|c|c|c|c|c|c|c|c|c|c|c|}
\hline 
n_1& n_2 & n_3 & n_4 & n_5 & n_6 & n_7 & n_8 & n_9 & n_{10} & n_{11} & n_{12} & n_{13} & n_{14} & n_{15} \\ 
\hline 
75 & 37 & 25 & 18 & 15 & 12 & 10 & 9 & 8 & 7 & 7 & 6 & 5 & 5 & 5 \\ 
\hline 
\end{array}
\end{equation}

For each pair $(P, A)$ and each (simple) eigenvalue $(\alpha_0, \beta_0)$ of $P(\alpha,\beta)$, we compute two quantities: the exact value of $Q_\theta^{p}$ (through the formula \eqref{Qrah} with the weights in Definition \ref{weights}(2)) and the upper bound on this quotient given in Theorem \ref{main-homo1}, which depends only on $\mathrm{cond}_{\infty}(A)$ and $Z_k$. These quantities are shown in the left plot of Figure \ref{fig1} as a function of $k$: the exact values of $Q_\theta^{p}$ are represented with the marker $*$  while the upper bounds use the marker $\circ$. Note that in this plot the scale of the vertical axis is logarithmic. This experiment confirms the (anticipated in Remark \ref{rem-hom-1}) fact that the factor $Z_k$ is very pessimistic, since we observe in the plot that, although the quotients $Q_\theta^p$ typically increase slowly with the degree $k$, they are much smaller than the corresponding upper bounds. A closer look at the exact values of $Q_\theta^{p}$ shows that most of them are larger than one, some considerably larger, and that the very few which are smaller than one are very close to one. We have observed this typical behavior of $Q_\theta^{p}$ (and also of $Q_\theta^{r}$) in all our {\em random} numerical experiments, but we stress that it is easy to produce tests with the  opposite behavior by interchanging the roles of $P$ and $M_A(P)$ and of $A$ and $A^{-1}$, respectively. Note that, in this case, the set of random matrix polynomials $M_A(P)$ is very different that the one produced by generating the matrix coefficients with the command {\texttt randn}.

We have performed an experiment similar to the one described in the previous paragraphs for confirming that $Z_k$ is also pessimistic in the bounds in Theorem \ref{main-homo} on $Q_\theta^r$. In this case, we have scaled the coefficients of the randomly generated matrix polynomials in such a way that the factor $\rho$ in \eqref{rho} is always equal to $10^3$. The plot for the obtained exact values of $Q_\theta^{r}$ and their upper bounds is essentially the one on the left of Figure \ref{fig1} with the vertical coordinates of all the   markers multiplied by $10^3$. For brevity, this plot is omitted.

\begin{figure}[ht]
	\centering
	\includegraphics[width=0.8\textwidth]{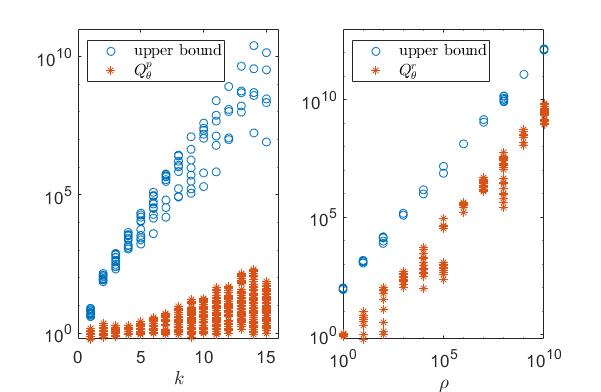}
	\caption{On the left results of Experiment 1, i.e., plot of $Q_\theta^p$ versus the degree $k$ for M\"{o}bius transformations induced by matrices $A$ with $\mathrm{cond}_{2}(A)=1$. On the right results of Experiment 2, i.e., plot of $Q_\theta^r$ versus $\rho$ for M\"{o}bius transformations  of matrix polynomials with degree $2$ induced by matrices $A$ with $\mathrm{cond}_{2}(A)=1$.}
	\label{fig1}
\end{figure}

\textbf{Experiment 2.} In this experiment, we have generated $30$ random matrix polynomials of size $5\times 5$ and degree $2$ for which the factor $\rho$ defined in \eqref{rho} equals  $10^t$, where $t$ has been randomly chosen for each polynomial by using the MATLAB's command {\texttt randi([0 10])}. More precisely, the matrix coefficients $B_0, B_1, B_2$ of these matrix polynomials  with $\rho =10^t$ have been generated with the next procedure. First, we generated matrix polynomials of size $5\times 5$ and degree $2$ by generating the matrix coefficients $B'_0, B'_1$, and $B'_2$ with MATLAB's command {\texttt randn}. For each of these polynomials, we determined $\rho_T:=\max \limits_{i=0:2}\{\|B'_i\|_2\}/\min\{\|B'_0\|_2,\|B'_2\|_2\}$ and the coefficient $B'_s$ such that $\|B'_s\|_2=\max \limits_{i=0:2}\{\|B'_i\|_2\}$. Then, the matrix coefficients $B_0', B_1'$ and $B_2'$ were scaled (obtaining new coefficients $B_0, B_1, B_2$) to get a new polynomial with the desired $\rho$, using the following criteria: If $\min\{\|B'_0\|_2,\|B'_2\|_2\}=\|B'_0\|_2$ and
\begin{itemize}
	\item[(a)]  $\|B'_0\|_2=\|B'_1\|_2=\|B'_2\|_2$, then  {\texttt q := randi([0 2])}, $B_q := \rho B'_q$ and $B_i := B'_i$ for $i\ne q.$
	\item[(b)]  $\|B'_0\|_2=\|B'_1\|_2=\|B'_2\|_2$ does not hold and $s=1$, then:
	\begin{itemize}
		\item[(b1)] If $\rho_T\leq\rho$, then $B_0: = \rho_T B'_0$, $B_1 := \rho B'_1$, and $B_2 := \rho_T B'_2$.
		\item[(b2)] If $\rho_T>\rho$, then $B_0 :=\rho_T B'_0$, $B_1 :=\rho B'_1$, and $B_2:= \rho (\|B'_1\|_2/\|B'_2\|_2) \,B'_2$.
	\end{itemize}
	\item[(c)] $\|B'_0\|_2=\|B'_1\|_2=\|B'_2\|_2$ does not hold and $s\ne 1$ (which means $s=2$), then:
	\begin{itemize}
		\item[(c1)] If $\rho_T\leq\rho$, then $B_0 := B'_0$, $B_1 := B'_1$, and $B_2 := (\rho/\rho_T) \, B'_2$.
		\item[(c2)] If $\rho_T>\rho$, then $B_0 := \rho_T B'_0$, $B_1 := \rho(\|B'_2\|_2/\|B'_1\|_2) \, B'_1$, and $B_2 := \rho B'_2$.
	\end{itemize}
\end{itemize}
If $\min\{\|B'_0\|_2,\|B'_2\|_2\}=\|B'_2\|_2$, then one proceeds in the same way but interchanging the roles of $B'_0$ and $B'_2$.

For each matrix polynomial $P$ generated as above, a random $2\times 2$ matrix $A$ with $\mathrm{cond}_{2}(A)=1$ was constructed as in Experiment 1 and, then,   $M_A (P)$ was computed. Finally, for each pair $(P,A)$ and each (simple) eigenvalue $(\alpha_0,\beta_0)$ of $P$, we computed two quantities: the exact value of $Q_\theta^r$, from the formula \eqref{Qrah} with the weights in Definition \ref{weights}(3), and the upper bound for this quotient in Theorem \ref{main-homo}, which depends only on $\mathrm{cond}_{\infty}(A)$, $\rho$, and $Z_2$. These quantities are shown in the right plot of Figure \ref{fig1} as a function of $\rho$: the markers of the exact values of $Q_\theta^{r}$ are $*$  and the markers of the upper bounds are $\circ$. Note that, in this plot, the scale of both the horizontal and vertical axes are logarithmic. It can be observed that many of the exact values of $Q_\theta^r$ essentially attain the upper bounds (recall that here $Z_2 = 72$), and, so, that $Q_\theta^r$ typically increases proportionally to $\rho$ for the random matrix polynomials that we have generated. We report that, if in this set of random polynomials the roles of $P$ and $M_A(P)$ and the roles of $A$ and $A^{-1}$ are interchanged, and the results are graphed against the factor $\widetilde{\rho}$ in \eqref{rho}, then the exact values of $Q_\theta^r$ essentially attain the {\em lower} bounds in Theorem \ref{main-homo}. This plot is omitted for brevity.

\textbf{Experiment 3.} In this experiment, we generated random matrix polynomials $P$ by generating their coefficients with MATLAB's command {\texttt randn}. In particular, we generated $30$ matrix polynomials of degree $k=1$ and sizes $5\times 5$, $10\times 10$, and $15\times 15$;  $20$ matrix polynomials of degree $k=2$ and sizes $5\times 5$ and $10\times 10$; and $20$ matrix polynomials of degree $k=3$ and sizes $5\times 5$ and $8\times 8$ (more precisely, $10$ matrix polynomials of each pair degree-size). For each polynomial $P$, a random $2\times 2$ matrix $A:=U\mbox{diag}(r , r/10^s)W$ was constructed, where
$U$ and $W$ are random orthogonal matrices generated as the unitary  Q matrices produced by the application of the MATLAB command {\texttt qr(randn(2))} twice; $r=${\texttt randn}, and $s$={\texttt randi([0 10])}, which implies $\mathrm{cond}_{2}(A)=10^s$.  Then the M\"{o}bius transform $M_A (P)$ of each polynomial $P$ was computed.

For each pair $(P,A)$ and each (simple) eigenvalue $(\alpha_0,\beta_0)$ of $P$, we computed two quantities: the exact value of $Q_\theta^{p}$ (from the formula \eqref{Qrah} with the weights in Definition \ref{weights}(2)) and $\mathrm{cond}_{\infty}(A)$. The quotients $Q_\theta^{p}$ are graphed (using the marker $*$) in the plots in Figure \ref{fig2} as a function of $\mathrm{cond}_{\infty}(A)$: the figure on the left  corresponds to the polynomials of degree $1$, the figure in the middle  corresponds to the polynomials of degree $2$, and the figure on the right  corresponds to the polynomials of degree $3$. Observe that in these plots the scales of both axes are logarithmic and that solid lines corresponding to the upper bounds in Theorem \ref{main-homo1} are also drawn. As announced and explained in Section \ref{subsec.53}, (recall, in particular, the fourth and fifth points) the differences between the behaviours of $Q_\theta^{p}$ for degrees $k = 1$ and $k \geq 2$ and the considered random polynomials are striking: typically, when $k=2$ or $k=3$, the exact values of $Q_\theta^p$ grow proportionally to $\mathrm{cond}_{\infty}(A)^{k-1}$  and are close to the upper bounds in Theorem \ref{main-homo1}, but, for $k=1$, $Q_\theta^p$  remains close to $1$ even when the matrix $A$ is extremely ill-conditioned. However, the reader should bear in mind that for any given matrix $A$, it is always possible (and easy) to construct regular matrix polynomials of degree $1$ (pencils) with eigenvalues for which the upper bound on $Q_\theta^p$ in Theorem \ref{main-homo1} is essentially attained, as it was explained in the third point in Section \ref{subsec.53}. We have generated pencils of this type but the results are not shown for brevity. Again, we report that, for degrees $2$ and $3$, if in these sets of random polynomials the roles of $P$ and $M_A(P)$ and the roles of $A$ and $A^{-1}$ are interchanged, then the exact values of $Q_\theta^p$ essentially attain the {\em lower} bounds in Theorem \ref{main-homo1}. These plots are also omitted for brevity.

\begin{figure}[ht]
	\centering
	\includegraphics[width=0.9\textwidth]{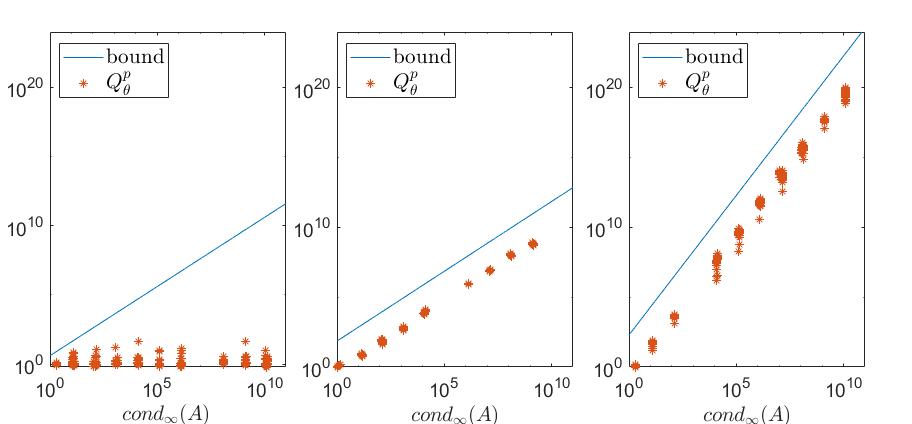}
	\caption{Results of Experiment 3: plots of $Q_\theta^p$ versus $\mathrm{cond}_{\infty}(A)$ for degrees $k=1$ (on the left), $k=2$ (on the middle), and $k=3$ (on the right).}
	\label{fig2}
\end{figure}

We performed an experiment analogous to Experiment 3 but where all matrix polynomials were generated so that the value of $\rho$ (as in \eqref{rho}) equaled  $10^3$.  The exact values of the quotients $Q_\theta^r$ and the upper bounds in Theorem \ref{main-homo} were then computed. The obtained plots are essentially the ones in Figure \ref{fig2} with the vertical coordinates of the quotients and the upper bounds multiplied by $10^3$.

\textbf{Experiment 4.} This experiment is the counterpart for backward errors of Experiment 1 and, as a consequence, is described very briefly. We generated a set of random matrix polynomials $P$ and  their M\"{o}bius transforms $M_A (P)$ exactly as in Experiment 1. Therefore, $\textrm{cond}_2(A)=1$ for all the matrices $A$ in this test. Then, for each pair $(P,A)$, we computed the (approximate) right eigenpairs  
of $M_A(P)(\gamma,\delta)$ in floating point arithmetic with the command {\tt polyeig}. For each of these computed eigenpairs, we computed two quantities:  $Q_{\eta,\mathrm{right}}^{p}$ (from the expression \eqref{eq.exactbackquot} with the weights in Definition \ref{def.weigths-back}(2)) and the upper bound on this quotient obtained in Theorem \ref{eta-bounds}, which depends only on  $\textrm{cond}_\infty(A)$ and $Y_k$. These quantities are shown in the left plot of Figure \ref{fig4} as functions of the degree $k$ of $P$. We observe the same behaviour as in the left plot of Figure \ref{fig1} and similar comments are valid. Therefore, it can be deduced that the factor $Y_k$ in the bounds on the quotients of the backward errors is very pessimistic. 

\textbf{Experiment 5.} This experiment is the counterpart of Experiment 2 for backward errors.  We generated a set of random matrix polynomials $P$ of degree $2$ and their M\"{o}bius transforms $M_A (P)$ exactly as in Experiment 2. For each pair $(P,A)$ and each right eigenpair  
of $M_A(P)(\gamma,\delta)$, computed in floating point arithmetic with {\tt polyeig}, two quantities are computed:  $Q_{\eta,\mathrm{right}}^{r}$ (from the expression \eqref{eq.exactbackquot} with the weights in Definition \ref{def.weigths-back}(3)) and the upper bound for this quotient in Theorem \ref{eta-bounds}, which depends only on $\textrm{cond}_\infty(A)$, $\rho$, and $Y_2$. These two quantities are shown in the right plot of Figure \ref{fig4} as functions of $\rho$. The same behaviour as in the right plot of Figure \ref{fig1} is observed and similar comments remain valid. Therefore, it can be deduced that the quotients $Q_{\eta,\mathrm{right}}^{r}$ of the backward errors  typically grow proportionally to $\rho$. 

Finally, we report that, for the quotients of backward errors $Q_{\eta,\mathrm{right}}^{p}$, we have also performed an experiment analogous to the Experiment 3. The corresponding plots are not presented in this paper for brevity. However, we stress that the plot corresponding to the degree $k=1$ is remarkably different from the left plot in Figure \ref{fig2}, since it shows that $Q_{\eta,\mathrm{right}}^{p}$ typically increases proportionally to $\mathrm{cond}_{\infty}(A)$  and, therefore, no difference of behavior is observed in this respect between the quotients of backward errors for degrees $k=1$ and $k\geq 2$. This fact was pointed out and explained in Remark \ref{remark-back}.

\begin{figure}[ht]
	\centering
	\includegraphics[width=0.8\textwidth]{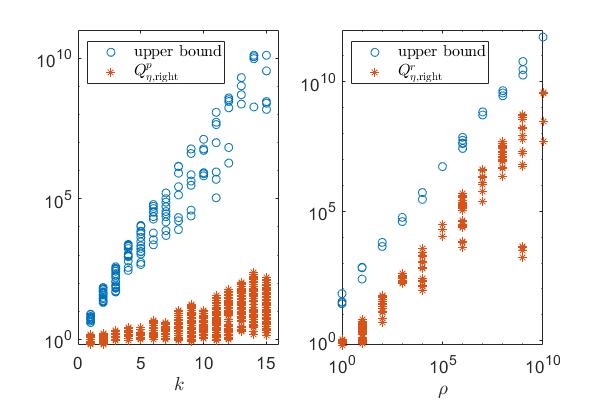}
	\caption{On the left results of Experiment 4, i.e., plot of $Q_{\eta,\mathrm{right}}^{p}$ versus the degree $k$ for M\"{o}bius transformations induced by matrices $A$ such that $\mathrm{cond}_2(A)=1$. On the right results of Experiment 5, i.e., plot of $Q_{\eta,\mathrm{right}}^{r}$ versus $\rho$ for M\"{o}bius transformations of matrix polynomials with degree $2$ induced by matrices $A$ such that $\mathrm{cond}_2(A)=1$.}
	\label{fig4}
\end{figure}

\section{Conclusions and future work} \label{sec.conclusions}

In this paper, we have studied the influence of M\"obius transformations on the (Stewart-Sun) eigenvalue condition number and backward errors of approximate eigenpairs of regular homogeneous matrix polynomials. More precisely,  we have given sufficient conditions, independent of the  eigenvalue, for the  condition number of a simple eigenvalue of a polynomial $P$ and the condition number of the associated eigenvalue   of a M\"obius transform of $P$ to be close. Similarly, we have given sufficient conditions for the backward error of an approximate eigenpair of a M\"obius transform of $P$ and the associated approximate eigenpair of $P$ to be close. In doing this analysis, we considered three variants of the Stewart-Sun condition number and of backward errors, depending on the selection of  weights involved in their definitions, that we called absolute, relative with respect to the norm of the polynomial, and relative. 

The most important conclusion of our study is that in the  relative-to-the-norm-of-the-polynomial case, if the matrix $A$ that defines the M\"obius transformation is well-conditioned and the degree of $P$ is moderate, then the M\"obius transformation preserves approximately the conditioning of the simple eigenvalues of $P$, and  the backward errors of the computed eigenpairs of $P$ are similar to the backward errors of the computed eigenpairs of $M_A(P)$. In the relative case, these conclusions  hold as well if, additionally, we assume that the matrix coefficients of $P$ (resp., the matrix coefficients  of   $M_A(P)$) have similar norms. Furthermore, we have provided some insight on the behavior of the quotients of eigenvalue condition numbers when the matrix $A$ defining the M\"obius transformation is ill-conditioned. Our study shows that, in this case, a significantly different typical behavior of the quotients of eigenvalue condition numbers can be expected when the matrix polynomial has degree 1, 2 or larger than 2. 

We must point out that the simple sufficient conditions for the approximate preservation of the eigenvalue condition numbers after the application of a M\"obius transformation to a homogeneous matrix polynomial cannot be immediately extrapolated to the non-homogeneous case, which will be studied in a separate paper. In this case, special attention must be paid to  eigenvalues with very large modulus or modulus close to $0$.

In this paper, we have only considered the effect of M\"obius transformations on the condition numbers of simple eigenvalues of a matrix polynomial. An interesting future line of research may be to extend our results to multiple eigenvalues by taking as starting point the condition numbers defined in \cite{mult1,mult2}. 

As explained in the introduction, in some relevant applications, the M\"obius transformations are used to compute invariant or deflating subspaces associated with eigenvalues with 
certain properties. Thus, studying how a M\"obius transformation affects the condition numbers of eigenvectors and invariant/deflating subspaces is an interesting problem that we will also address separately.


\begin{thebibliography}{99}
	\bibitem{comparison} {\sc L.M. Anguas, M.I. Bueno, and F.M. Dopico}, \emph{A comparison of eigenvalue condition numbers for matrix polynomials}, submitted (available in arXiv:1804.09825). 
	\bibitem{vandermonde} {\sc W.N. Bailey}, \emph{Generalised Hypergeometric Series}, Cambridge University Press, Cambridge, 1935.
	\bibitem{benneretal} {\sc P. Benner, V. Mehrmann, and H. Xu}, \emph{A numerically stable, structure preserving method for computing the eigenvalues of real Hamiltonian
		or symplectic pencils}, Numer. Math., 78 (1998), 329--358.
	\bibitem{NLEVP-2013} {\sc T. Betcke, N.J. Higham, V. Mehrmann, C. Schr\"oder, and F. Tisseur}, \emph{NLEVP: A collection of nonlinear eigenvalue problems}, ACM Trans. Math. Software, 39 (2013), 7:1-7:28.
	\bibitem{brualdi}{\sc R.A. Brualdi}, \emph{Introductory Combinatorics}, Prentice Hall, New Jersey, 2010.
	\bibitem{bueno-calcolo-2018} {\sc M.I. Bueno, F.M. Dopico, S. Furtado, and L. Medina}, \emph{A block-symmetric linearization of odd degree matrix polynomials with optimal eigenvalue condition number and backward error}, Calcolo, 55 (2018), 32:1-32:43. 
	\bibitem{buenoetal} {\sc M.I. Bueno, K. Curlett, and S. Furtado}, \emph{Structured linearizations from Fiedler pencils with repetition I.}, Linear Algebra Appl., 460 (2014), 51-80.
	\bibitem{byers} {\sc R. Byers, V. Mehrmann, and H. Xu}, \emph{A structured staircase algorithm for skew-symmetric/symmetric pencils}, Electron. Trans. Numer. Anal., 26 (2007), 1-33.
	\bibitem{Ded-first} {\sc J.P. Dedieu}, \emph{Condition operators, condition numbers, and condition number theorem for the generalized eigenvalue problem}, Linear Algebra Appl., 263 (1997), 1-24.
	\bibitem{DedTis2003} {\sc J.P. Dedieu and F. Tisseur}, \emph{Perturbation theory for homogeneous polynomial eigenvalue problems}, Linear Algebra Appl., 358 (2003), 71-94.
	\bibitem{dtdomack}  {\sc F. De Ter\'an, F.M. Dopico, and D.S. Mackey}, \emph{Palindromic companion forms for matrix polynomials of odd degree}, J. Comput. Appl. Math., 236 (2011), 1464-1480.
	\bibitem{dtdopvd2015} {\sc F. De Ter\'an, F.M. Dopico, and  P. Van Dooren},
	\emph{Matrix  polynomials  with  completely prescribed eigenstructure},
	SIAM J. Matrix Anal. Appl., 36 (2015), 302-328.
	\bibitem{block-kron}   {\sc F.M. Dopico, P.W. Lawrence, J. P\'erez, and P. Van Dooren},
	\emph{Block Kronecker linearizations of matrix polynomials and their backward errors}, Numer. Math., 140 (2018), 373-426.
	\bibitem{dopevd} {\sc F.M. Dopico, J. P\'erez, and P. Van Dooren}, \emph{Structured backward error analysis of linearized structured polynomial eigenvalue problems}, Math. Comp. (2018), {\tt https://doi.org/10.1090/mcom/3360}.
	\bibitem{FanLinVD} {\sc H.-Y. Fan, W.-W. Lin, and P. Van Dooren}, \emph{Normwise scaling of second order polynomial matrices}, SIAM J. Matrix Anal. Appl., 26 (2004), 252-256.
	\bibitem{hamarling} {\sc S. Hammarling, C.J. Munro, and F. Tisseur}, \emph{An algorithm for the complete solution of quadratic eigenvalue problems}, ACM Trans. Math. Software, 39 (2013), 18:1-18:19.
	\bibitem{higham-function-book} {\sc N.J. Higham}, \emph{Functions of Matrices: Theory and Computation}, SIAM, Philadelphia, 2008.
	\bibitem{HigLiTis}{\sc N.J. Higham, R.-C. Li, and F. Tisseur}, \emph{Backward error of polynomial eigenproblems solved by linearization}, SIAM J. Matrix Anal. Appl., 29 (2007), 1218-1241.
	\bibitem{HigMacTis} {\sc N.J. Higham, D.S. Mackey, and F. Tisseur}, \emph{The conditioning of linearizations of matrix polynomials}, SIAM J. Matrix Anal. Appl., 28 (2006), 1005-1028.
	\bibitem{mult1} {\sc D. Kressner, M.J. Pel\'aez, and J. Moro}, \emph{Structured H\"{o}lder condition numbers for multiple eigenvalues}, SIAM J. Matrix Anal. Appl., 31 (2009), 175-201.
	\bibitem{lan-rod} {\sc P. Lancaster and L. Rodman}, \emph{The Algebraic Riccati Equation}, Oxford University Press, Oxford, 1995.
	\bibitem{good} {\sc D.S. Mackey, N. Mackey, C. Mehl, and V. Mehrmann}, \emph{Structured polynomial eigenvalue problems: good vibrations from good linearizations}, SIAM J. Matrix Anal. Appl., 28 (2006), 1029-1051.
	\bibitem{Mobius} {\sc D.S. Mackey, N. Mackey, C. Mehl, and V. Mehrmann}, \emph{ M\"{o}bius transformations of matrix polynomials},  Linear Algebra Appl., 470 (2015), 120-184.
	\bibitem{MCM1} {\sc B. McMillan}, \emph{Introduction  to  formal  realizability  theory.  I}, Bell System Tech. J., 31 (1952), 217-279.
	\bibitem{MCM2} {\sc B. McMillan}, \emph{Introduction to formal realizability theory. II}, Bell System Tech. J., 31 (1952), 541-600.
	\bibitem{Mehr91} {\sc V. Mehrmann}, \emph{The Autonomous Linear Quadratic Control Problem, Theory and Numerical Solution}, Lecture Notes in Control and Information Sciences, vol. 163, Springer-Verlag, Heidelberg, 1991.
	\bibitem{cayley} {\sc V. Mehrmann}, \emph{A step toward a unified treatment of continuous and discrete time control problems}, Linear Algebra Appl., 241-243 (1996), 749-779.
	\bibitem{MehrPol} {\sc V. Mehrmann and  F. Poloni}, \emph{A generalized structured doubling algorithm for the numerical solution of linear quadratic optimal control problems}, Numer. Linear Algebra Appl., 20 (2013), 112-137.
	\bibitem{MehrXu} {\sc V. Mehrmann and  H. Xu}, \emph{Structure preserving deflation of infinite eigenvalues in structured pencils}, Electron. Trans. Numer. Anal., 44 (2015), 1--24.
	\bibitem{mult2} {\sc N. Papathanasiou and P. Psarrakos}, \emph{On condition numbers of polynomial eigenvalue problems}, Appl. Math. Comp., 216 (2010), 1194-1205. 
	\bibitem{Stewart} {\sc G.W. Stewart and J.-G. Sun}, \emph{Matrix Perturbation Theory},  Academic Press,  Boston, 1990.
	\bibitem{Tis2000}{\sc  F. Tisseur}, \emph{ Backward error and condition of polynomial eigenvalue problems}, Linear Algebra Appl., 309 (2000), 339-361.
	\bibitem{TisZab} {\sc F. Tisseur and  I. Zaballa}, \emph{Triangularizing quadratic matrix polynomials}, SIAM J. Matrix Anal. Appl., 34 (2013), 312-337.
	\bibitem{VanBarel-Tisseur} {\sc M. Van Barel and F. Tisseur}, \emph{Polynomial eigenvalue solver based on tropically scaled Lagrange linearization}, Linear Algebra Appl., 542 (2018), 186–-208.
	\bibitem{VD-DW} {\sc P. Van Dooren and P. Dewilde}, \emph{The eigenstructure of an arbitrary polynomial matrix: computational aspects}, Linear Algebra Appl., 50 (1983), 545-579.
	\bibitem{weyl} {\sc H. Weyl}, {\em The Classical Groups}, Princeton University Press, Princeton, NJ, 1973.
	\bibitem{zeng-su} {\sc L. Zeng and Y. Su}, \emph{A backward stable algorithm for quadratic eigenvalue problems}, SIAM J. Matrix Anal. Appl., 35 (2014),  499-516.
\end{thebibliography}
\end{document}